\documentclass[a4paper]{amsart}
\usepackage{amssymb,amsfonts}
\usepackage[all,arc]{xy}
\usepackage{enumerate}
\usepackage{mathrsfs}
\usepackage{tikz}
\usepackage{tikz-cd}
\usetikzlibrary{arrows,snakes,backgrounds}
\tikzset{>=stealth}
\usepackage{color}   
\usepackage{marginnote}
\usepackage{nameref}
\usepackage[backref=page]{hyperref}
\usepackage{MnSymbol} 
\hypersetup{
    colorlinks=true, 
    linktoc=all,     
    citecolor=black,
    filecolor=black,
    linkcolor=blue,
    urlcolor=black
}
\usepackage{cleveref}
\usetikzlibrary{arrows}

\newtheorem{thm}{Theorem}[section]
\newtheorem{cor}[thm]{Corollary}
\newtheorem{prop}[thm]{Proposition}
\newtheorem{lem}[thm]{Lemma}

\theoremstyle{definition}
\newtheorem{defn}[thm]{Definition}

\newtheorem{con}[thm]{Construction}

\theoremstyle{remark}
\newtheorem{rem}[thm]{Remark}

\newcommand{\bF}{\mathbb{F}}

\newcommand{\bH}{\mathbb{H}}

\newcommand{\bQ}{\mathbb{Q}}
\newcommand{\bR}{\mathbb{R}}

\newcommand{\bZ}{\mathbb{Z}}

\newcommand\Diff{\mathrm{Diff}}
\newcommand\BDiff{\mathrm{BDiff}}
\newcommand\dDiff{\mathrm{Diff}^{\delta}}

\newcommand\BdDiff{\mathrm{BDiff}^{\delta}}

\newcommand{\hcoker}{/\!\!/}
\newcommand{\hker}{\backslash\!\!\backslash}

\newcommand{\SnSn}{\#_g S^n \times S^n}
\newcommand{\tE}{\text{E}}
\newcommand{\tW}{\text{W}}
\newcommand{\W}{\text{W}_{g,1}}
\newcommand{\WW}{\text{\textnormal{W}}_{g,1}}
\newcommand{\ttW}{\text{\textnormal{W}}}

\addtocontents{toc}{\protect\setcounter{tocdepth}{1}}
\makeatletter
\let\c@equation\c@thm
\makeatother
\numberwithin{equation}{section}
\title{Homological stability and Stable Moduli of Flat manifold bundles}

\author{Sam Nariman}
\email{nariman@uni-muenster.de}
\address{Mathematical Institute\\
Universit{\"a}t M{\"u}nster\\
Einsteinstr. 62\\
D-48149 M{\"u}nster}

\begin{document}

\begin{abstract}
 We prove that the group homology of the diffeomorphism group of $\#_g S^n \times S^n\backslash \text{int}(D^{2n})$ as a discrete group is independent of $g$ in a range, provided that $n>2$. This answers the high dimensional version of a question posed by Morita about surface diffeomorphism groups made discrete.  The stable homology is isomorphic  to the homology of a certain infinite loop space related to the Haefliger's classifying space of foliations. One geometric consequence of this description of the stable homology is a splitting theorem that implies certain classes called generalized Mumford-Morita-Miller classes  can be detected on flat $(\#_g S^n \times S^n)$-bundles for $g\gg 0$.

\end{abstract}

\maketitle

\tableofcontents

\section{Introduction}
\subsection{Statements of the main results} We begin by fixing some notations appearing in this paper. For a manifold $\mathrm{M}$,  let $\Diff (\mathrm{M}, \partial)$  denote the topological group of $C^\infty$-diffeomorphisms   whose supports are away from $\partial \mathrm{M}$  with the $C^\infty$-topology. The same underlying group equipped with the $\it{discrete}$ topology is denoted by $\mathrm{Diff}^{\delta}(\mathrm{M},\partial)$. 

The purpose of this paper is to study a high dimensional version of a homological stability problem for discrete surface diffeomorphisms posed in \cite[Problem 45]{Morita}. To define higher dimensional analogue of Morita's problem, for $n>2$, we write
 \[
 \text{W}_{g,k}=(\SnSn)\backslash\coprod_1^k \text{int}(D^{2n})
 \]
 which is a manifold with  boundary obtained from the $g$-fold connected sum $\SnSn$  by cutting out the interior of $k$ disjoint disks. We apply Quillen's stability machine to prove homological stability for diffeomorphism groups of $\text{W}_{g,1}$ made discrete. We study the surface diffeomorphism groups in a separate paper.
 

   Let $j: \W\hookrightarrow \tW_{g+1,1}$ be an embedding such that the complement of the interior of $j(\W)$ in $  \tW_{g+1,1}$ is diffeomorphic to $\text{W}_{1,2}$. Having fixed such an embedding, one can define a homomorphism $ \dDiff(\W,\partial)\rightarrow \dDiff(\tW_{g+1,1},\partial)$ by extending diffeomorphisms via identity on the complement of $j(\W)$. Although this homomorphism depends on $j$, any two choices of embeddings lead to conjugate homomorphisms; therefore, we obtain a well-defined map up to homotopy between classifying spaces $\BdDiff(\W,\partial)\rightarrow \BdDiff(\tW_{g+1,1},\partial)$.  Our first main theorem is the following
   
    \begin{thm}\label{MainTheorem}
  For $n>2$ the stabilization map
  \begin{equation*}
   H_k(\BdDiff(\ttW_{g,1},\partial);\bZ)\rightarrow H_k(\BdDiff(\ttW_{g+1,1},\partial);\bZ)
  \end{equation*}
  is  surjection for $k\leq (g-2)/{2}$ and an isomorphism for $k< (g-2)/2$.
 \end{thm}

\begin{rem}
  If we denote $C^1$-diffeomorphisms of $\W$ by $\Diff^1(\W,\partial)$, one consequence of  Tsuboi's remarkable theorem  \cite{tsuboi1989foliated} and the h-principle theorem of Thurston \cite{thurston1974foliations} is that ${\mathrm {BDiff}^{\delta,1}}(\W,\partial)$ is homology equivalent to $\BDiff^1(\W,\partial)$. But it is well known that the group $\Diff^1(\W,\partial)$ with the $C^1$-topology is weakly equivalent to the $C^{\infty}$-diffeomorphism group $\Diff(\W,\partial)$, therefore $\BDiff^1(\W,\partial)\simeq \BDiff(\W,\partial)$. Hence, the homological stability for $C^1$-diffeomorphisms with the discrete topology, ${\mathrm {Diff}^{\delta,1}}(\W,\partial)$ is already implied by the homological stability of $\BDiff(\W,\partial)$  which was proved by Galatius and Randal-Williams \cite{Galatius-Randal-Williams}.
\end{rem}
\begin{rem}In fact one can use the Mather-Thurston theory as we explained in \Cref{stablehomology} and the homological stability for tangential structures in \cite[Theorem 1.4 (ii)]{galatius2014homological} to give a short proof of \Cref{MainTheorem} along the lines of the sequel to this paper \cite[Section 1.2.3]{nariman2015stable}. But instead of giving a high powered proof of this theorem, we show that one can in fact prove the homological stability of $\dDiff(\W,\partial)$ without resorting to the deep theorem of Mather and Thurston.
\end{rem}

 Theorem \ref{limit homology} below describes the stable homology of these diffeomorphism groups with discrete topology  as the homology of an infinite loop space, which we now describe.  Let $\Gamma_{2n}$ be the Haefliger category, i.e. the topological groupoid whose objects are points in $\mathbb{R}^{2n}$ with its usual topology and morphisms between two points, say $x$ and $y$, are germs of diffeomorphisms that send $x$ to $y$ with the sheaf topology (for details see  \cite{haefliger1971homotopy}). The classifying space of this groupoid plays an important role in classifying foliations up to concordance.  By $\mathrm{S}\Gamma_{2n}$, we mean the subcategory of $\Gamma_{2n}$ with the same objects, but whose morphisms are orientation preserving diffeomorphisms. Its classifying space $\mathrm{BS}\Gamma_{2n}$, classifies  Haefliger structures together with an orientation on their normal bundles up to concordance. Sending a germ of a diffeomorphism to its derivative induces a map  
   \[
 \begin{tikzpicture}[node distance=2.2cm, auto]
  \node (A) {$\nu: \mathrm{BS}\Gamma_{2n}$};
  \node (B) [right of=A] {$\mathrm{BGL}_{2n}^+(\bR)$, };
  \draw [->] (A) to node {$$}(B);
\end{tikzpicture}
\]
where $\mathrm{GL}_{2n}^+(\bR)$ denotes the group of real matrices with positive determinants.  This map $\nu$ classifies the normal bundle to the universal  Haefliger structure on $\mathrm{BS}\Gamma_{2n}$. Let $\mathrm{BGL}_{2n}^+(\bR)\langle n\rangle$ be the $n$-connected cover of $\mathrm{BGL}_{2n}^+(\bR)$. Given  that $\nu$ is a $(2n+2)$-connected map \cite[Remark 1]{haefliger1971homotopy}, we have the following homotopy pullback diagram
   \[
 \begin{tikzpicture}[node distance=2cm, auto]
  \node (A) {$\mathrm{BS}\Gamma_{2n}$};
  \node (B) [below of=A] {$\mathrm{BGL}_{2n}^+(\bR)$.};
   \node (C) [left of= A, node distance=3.2cm] {$\mathrm{BS}\Gamma_{2n}\langle n\rangle$};
  \node (D) [below of=C] {$\mathrm{BGL}_{2n}^+(\bR)\langle n\rangle$};
  \draw [->] (C) to node {$\nu\langle n\rangle$}(D);
  \draw [->] (A) to node {$\nu$}(B);
  \draw [->] (C) to node {$\theta$}(A);
  \draw [->] (D) to node {$\theta^n$}(B);
\end{tikzpicture}
\]

 Take the inverse of the tautological bundle, $-\gamma$, on $\mathrm{BGL}_{2n}^+(\bR)$ and pull it back to $\mathrm{BS}\Gamma_{2n}\langle n\rangle$ via $\theta\circ \nu$. 
 We denote the Thom spectrum of this virtual bundle by $\bold{MT}\nu^n$, for a definition of the Thom spectrum of a virtual bundle see e.g. \cite[12.29]{switzer1975algebraic}. 
We shall write $\Omega^{\infty}\bold{MT}\nu^n$ for the associated infinite loop space and $\Omega_0^{\infty}\bold{MT}\nu^n$ for its base point component. In \Cref{stablehomology}, we use a model for $ \BdDiff(\text{W}_{g,1},\partial)$ to construct a map
   \[
 \begin{tikzpicture}[node distance=3.2cm, auto]
  \node (A) {$\alpha: \BdDiff(\text{W}_{g,1},\partial)$};
  \node (B) [right of=A] {$\Omega_0^{\infty}\bold{MT}\nu^n$.};
  \draw [->] (A) to node {$$}(B);
\end{tikzpicture}
\]
Our second main theorem is
 \begin{thm}\label{limit homology}
 The map induced by $\alpha$ on the homology
    \[
 \begin{tikzpicture}[node distance=4.2cm, auto]
  \node (A) {$H_k(\BdDiff(\ttW_{g,1},\partial);\bZ)$};
  \node (B) [right of=A] {$H_k(\Omega_0^{\infty}\bold{MT}\nu^n;\bZ)$,};
  \draw [->] (A) to node {$$}(B);
\end{tikzpicture}
\]
is a surjection for $k\leq (g-2)/{2}$ and an isomorphism for $k<(g-2)/2$.

 \end{thm}
 
 \subsection{Applications} The identity homomorphism induces a map
 $$\iota: \BdDiff(\W,\partial)\rightarrow \BDiff(\W,\partial).$$ As applications of \Cref{MainTheorem} and \Cref{limit homology}, we prove two results  about the map induced by $\iota$ on the  cohomology in the stable range. 
 
 Let $\bold{MT}\theta^n$ be the Thom spectrum of the virtual bundle $(\theta^n)^*(-\gamma)$ over the base $\mathrm{BGL}_{2n}^+(\bR)\langle n\rangle$ and let $\Omega_0^{\infty}\bold{MT}\theta^n$ be the base point component of the infinite loop space associated to this spectrum.  Galatius and Randal-Williams showed in \cite{Galatius-Randal-Williams} that even integrally the stable cohomology of topologized diffeomorphisms of $\W$ is isomorphic to the cohomology of $\Omega_0^{\infty}\bold{MT}\theta^n$. One can easily see that the rational cohomology of $\Omega_0^{\infty}\bold{MT}\theta^n$ is a polynomial algebra generated by certain classes $\kappa_c$ that are called generalized MMM classes (see \Cref{remarks} for definitions). 

 Unlike the description of the stable cohomology of topologized diffeomorphisms, it is not easy to compute $H^*(\Omega_0^{\infty}\bold{MT}\nu^n;\bQ)$. In order to construct nontrivial classes in the stable cohomology of $\dDiff(\text{W}_{g,1},\partial)$, one might attempt to pull back  generalized MMM classes from the cohomology of $ \BDiff (\text{W}_{g,1},\partial)$. In rational cohomology, $H^*(\BdDiff (\text{W}_{g,1},\partial);\bQ)$, the pullbacks of the generalized MMM classes vanish in degrees larger than $4n$ according to a result of Bott (see Section 6). We show that the situation is very different for cohomology with finite field coefficients.  
     \begin{thm}\label{splitting}
   For any prime $p$, the natural map
   \[
   \begin{tikzpicture}[node distance=3.9cm, auto]
  \node (A) {$ H^*(\Omega_0^{\infty}\bold{MT}\theta^n;\bF_p)$};
  \node (B) [right of=A] {$ H^*(\Omega_0^{\infty}\bold{MT}\nu^n;\bF_p)$};
  \draw [right hook->] (A) to node {$$}(B);
\end{tikzpicture}
   \]
    is split injective.
  \end{thm}
  \begin{cor}
  For any prime $p$, the natural map
\[  \begin{tikzpicture}[node distance=4.5cm, auto]
  \node (A) {$H^*(\BDiff(\WW,\partial);\bF_p)$};
  \node (B) [right of=A] {$H^*(\BdDiff(\WW,\partial);\bF_p)$};
  \draw [right hook->] (A) to node {$$}(B);
\end{tikzpicture}
\]
is injective, provided that $*< (g-2)/2$.
  \end{cor}
  This corollary implies that in the stable range, those MMM-classes that are nontrivial in  $H^*(\BDiff(\WW,\partial);\bZ)$ are also nontrivial in $H^*(\BdDiff(\WW,\partial);\bZ)$. Akita, Kawazumi and Uemura in \cite{MR1829309} used finite group actions on surfaces to prove the independence of MMM-classes for surface bundles. Their method also proves that there are flat surface bundles with nontrivial MMM-classes. In \Cref{sec6}, we use a Lie group action on $\tW_{g}$ described in \cite{galatius2015tautological} and a theorem of Milnor in \cite{milnor1983homology} to prove that  certain MMM-classes are nontrivial even in the unstable range of $H^*(\BdDiff(\WW,\partial);\bZ)$.
   
   The group homology of $\dDiff(\text{W}_{g,1},\partial)$ with integer or rational coefficients is believed to be gigantic and although  there  is not much known about the cohomology of $\Omega_0^{\infty}\bold{MT}\nu^n$,   \Cref{limit homology} implies that there are nontrivial cohomology classes arising from secondary characteristic classes of foliations known as Godbillon-Vey classes that vary continuously. More precisely, in \Cref{remarks} we show
  \begin{cor}\label{GV}
  For  $4n+4\leq g$, there is a surjection
    \[
 \begin{tikzpicture}[node distance=3.3cm, auto]
  \node (A) {$H_{2n+1}(\BdDiff (\ttW_{g,1},\partial);\bQ)$};
  \node (B) [right of=A] {$\bR^{v_{2n}}$};
  \draw [->>] (A) to node {$$}(B);
\end{tikzpicture}
\]
where $v_{2n}$ denotes the size  of a certain set of secondary characteristic classes $V_{2n}$ in $H_{4n+1}(\mathrm{BS}\Gamma_{2n}; \bZ)$ (see \cite[Remark 2.4] {MR769761} for detailed description of the set $V_{2n}$). For $n>2$, the number $v_{2n}$ is at least $3$.
  \end{cor}
  \begin{rem}\label{rem1}
  It is easy to use \Cref{GV} and {\it discontinuous invariants} as in \cite{MR877332}, to show that for $k\leq \frac{g-2}{4n+2}$, the group $H_{(2n+1)k}(\BdDiff (\ttW_{g,1},\partial);\bQ)$ is nontrivial (see \Cref{rem6.11} for a precise statement).
  \end{rem}
\subsection*{Outline of the paper}  This paper is organized as follows: in \Cref{maps}, we discuss various models of the stabilization maps and prove that although they are not homotopic, they induce  homology isomorphisms in the same range. In \Cref{resolution}, we construct a highly connected semisimplicial set on which $\dDiff(\W, \partial)$ acts and we determine the set of orbits of this action. In \Cref{spectralsequence}, we use the ``relative" spectral sequence argument in the sense of \cite{charney1987generalization} to establish homological stability. In \Cref{stablehomology},  we apply Thurston's theorem about classifying foliations to prove \Cref{limit homology} and we use a transfer argument to prove our splitting \Cref{splitting}. In \Cref{remarks}, we discuss various applications of \Cref{MainTheorem}, \Cref{limit homology}, and \Cref{splitting} to obtain partial results about characteristic classes of flat $\W$-bundles. 
\section*{Acknowledgments} I would like to thank my advisor, S\o ren Galatius, for proposing this problem. Without his help and his encouragement, this article would have  never existed. I would also like to thank Oscar Randal-Williams for many helpful discussions, in particular I learned  the method of ``relative" spectral sequence from him. I also owe Cary Malkiewich for his helpful comments on the proof of \Cref{split}. I would like to thank Steve Hurder for pointing out to me that there are more continuously varying secondary classes than I wrote in the first draft. I would also like to thank Alexander Kupers and the referee for their careful reading and detailed comments on the first draft that made me considerably improve  the paper.  This work was supported in part by NSF grants DMS-1105058 and DMS-1405001.
\section{Stabilization maps induce same map on homology} \label{maps}
For  reasons that will become clear in \Cref{resolution} and \Cref{spectralsequence}, it is convenient to work with a stabilization map that is different from the one defined in the introduction.  In this section, we describe this nonstandard stabilization map and prove it induces a homology isomorphism in all degrees that  the standard stabilization does. 
In the introduction, we defined $\tW_{g,k}$ which is well-defined up to diffeomorphism. We make our choices once and for all and let the notation $\tW_{g,k}$ denote the actual abstract manifold instead of a diffeomorphism class.

  \begin{figure}[h]

\begin{tikzpicture}[scale=.4]
\begin{scope}[shift={(0,7)}]
\begin{scope}[shift={(7,0)}]
\draw[line width=1.05pt] [dashed] (0,.-2.5) arc (-90:90:0.5 and 2.5);
\draw [line width=1.05pt] (0,2.5) arc (90:270:0.5 and 2.5);
\draw [line width=1.05pt] (0,2.5)--(13,2.5);
\draw [line width=1.05pt] (13,-2.5) arc (-90:90:2.5);
\draw [line width=1.05pt] (0,-2.5)--(13,-2.5);
\draw [line width=1.1pt] (4.6,.5) arc (240:300:2.5 and 6.75);
\draw   [line width=1.7pt] (6.87, 0.13) arc (48:150:1.2 and 0.9);
\begin{scope}[shift={(-3,0)}]
\draw [line width=1.05pt] (4.6,.5) arc (240:300:2.5 and 6.75);
\draw   [line width=1.7pt] (6.87, 0.13) arc (48:150:1.2 and 0.9);
\end{scope}
\begin{scope}[shift={(6,0)}]
\draw [line width=1.05pt] (4.6,.5) arc (240:300:2.5 and 6.75);
\draw   [line width=1.7pt] (6.87, 0.13) arc (48:150:1.2 and 0.9);
\end{scope}
\node at  (8,0) {$\bullet$};
\node at (9,0) {$\bullet$};
\node at (10,0) {$\bullet$};
\end{scope}

\draw[line width=1.05pt] [dashed] (-1.2,.-2.5) arc (-90:90:0.5 and 2.5);
\draw [line width=1.05pt] (-1.2,2.5) arc (90:270:0.5 and 2.5);
\draw [line width=1.05pt] (-1.2,2.5)--(3,2.5);
\draw [line width=1.05pt] (-1.2,-2.5)--(3,-2.5);
\node at (-5,0) {a)};
\begin{scope}[shift={(-5,0)}]
\draw [line width=1.05pt] (4.6,.5) arc (240:300:2.5 and 6.75);
\draw   [line width=1.7pt] (6.87, 0.13) arc (48:150:1.2 and 0.9);
\end{scope}
\begin{scope}[shift={(3,0)}]
 \draw [line width=1.05pt] (0,0) ellipse (0.5 and 2.5);
 \draw [<->,thick] (1.5,0)--(2.5,0);
 \node at (2,-1) {glue};
\end{scope}
\node at (1,-3.3) {$\text{W}_{1,2}$};
\node at (13.3,-3.3) {$\text{W}_{g,1}$};
\end{scope}

\begin{scope}[shift={(7,0)}]
\draw[line width=1.05pt] [dashed] (0,.-2.5) arc (-90:90:0.5 and 2.5);
\draw [line width=1.05pt] (0,2.5) arc (90:270:0.5 and 2.5);
\draw [line width=1.05pt] (0,2.5)--(13,2.5);
\draw [line width=1.05pt] (13,-2.5) arc (-90:90:2.5);
\draw [line width=1.05pt] (0,-2.5)--(13,-2.5);
\draw [line width=1.1pt] (4.6,.5) arc (240:300:2.5 and 6.75);
\draw   [line width=1.7pt] (6.87, 0.13) arc (48:150:1.2 and 0.9);
\begin{scope}[shift={(-3,0)}]
\draw [line width=1.05pt] (4.6,.5) arc (240:300:2.5 and 6.75);
\draw   [line width=1.7pt] (6.87, 0.13) arc (48:150:1.2 and 0.9);
\end{scope}
\begin{scope}[shift={(6,0)}]
\draw [line width=1.05pt] (4.6,.5) arc (240:300:2.5 and 6.75);
\draw   [line width=1.7pt] (6.87, 0.13) arc (48:150:1.2 and 0.9);
\end{scope}
\node at  (8,0) {$\bullet$};
\node at (9,0) {$\bullet$};
\node at (10,0) {$\bullet$};
\end{scope}
\draw [line width=1.05pt] (-0.5,2.5) arc (90:270:2.5);
\draw [line width=1.05pt] (3,2.5)--(5,2.5);
\draw [line width=1.05pt] (2.5,1.1)--(4.6,1.1);
\draw [line width=1.9pt] (5,2.5) arc (90:155:0.5 and 2.5);
\draw [line width=1.9pt] (7,2.5) arc (90:155:0.5 and 2.5);
\node at (-5,0) {b)}; 
\draw [line width=1.05pt] (-0.5,2.5)--(3,2.5);
\draw [line width=1.05pt] (-0.5,-2.5)--(3,-2.5);
\begin{scope}[shift={(-5.8,0)}]
\draw [line width=1.05pt] (4.6,.5) arc (240:300:2.5 and 6.75);
\draw   [line width=1.7pt] (6.87, 0.13) arc (48:150:1.2 and 0.9);
\end{scope}
\begin{scope}[shift={(3,0)}]
 \draw [line width=1.05pt] (0,2.5) arc (90:270:0.5 and 2.5);
  \draw [line width=1.05pt] (0.5,1.1) arc (25:-90:0.5 and 2.5);
 \draw [line width=1.05pt][dashed] (0,2.5) arc (90:25:0.5 and 2.5);
 \draw [<->,thick] (2,1.6)--(3,1.6);
\node at (2.5,1) {\small glue};
\end{scope}
\node at (1,-3.3) {$\text{H}$};
\node at (13.3,-3.3) {$\text{W}_{g,1}$};

\end{tikzpicture} 
\caption{a) Standard stabilization map, b) Non-standard stabilization map, for $n=1$}
\label{stabilization}
\end{figure}
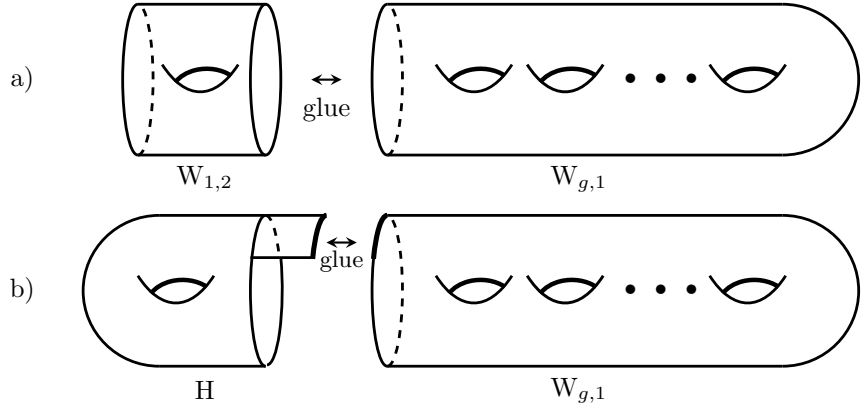

 \subsection*{Standard stabilization maps} Let $j:\tW_{g,1} \hookrightarrow \tW_{g+1,1}$ be an embedding so that $\tW_{g+1,1}\backslash \text{int}(j(\tW_{g,1}))$ is diffeomorphic to $\tW_{1,2}$. This embedding is unique up to diffeomorphisms of $\tW_{g+1,1}$ and extending diffeomorphisms  of $\W$ via the identity on the complement of $j(\W)$ induces a homomorphism $s_j$ from  $\dDiff(\tW_{g,1},\partial)$ to  $\dDiff(\tW_{g+1,1},\partial)$. This injective map is unique up to conjugation and it induces a well-defined map, up to homotopy, from $\BdDiff(\W,\partial)$ to $\BdDiff(\tW_{g+1,1},\partial)$.

 
\subsection*{Non-standard stabilization maps} Another model for a stabilization map that is not conjugate to $s_j$  is described as follows. The boundary connect sum of $\W$ and $\tW_{1,1}$ is diffeomorphic to $\tW_{g+1,1}$. Let   $f :\W\natural \tW_{1,1}\to\tW_{g+1,1}$ be a diffeomorphism. The map $f$ identifies $\W$ with a submanifold of $\tW_{g+1,1}$. Using this identification, we can define a nonstandard stabilization map $ns_{f}:\dDiff(\W,\partial)\to \dDiff(\tW_{g+1,1},\partial)$ that extends elements of $\dDiff(\W,\partial)$ by the identity. 

\begin{defn}[\emph{Pushing collar map}] For a manifold $\text{W}$ with a nonempty boundary, let $[0,1)\times \partial \tW\hookrightarrow \tW$ be a fixed collar neighborhood of the boundary. For $\epsilon <1/2$, let $\lambda_{\epsilon}:[0,1]\to [0,1]$ be an embedding such that $\lambda_{\epsilon}(t)=t/2+\epsilon$ for $t\leq 1.9\epsilon$ and $\lambda_{\epsilon}(t)=t$ for $t\geq 2\epsilon$.  
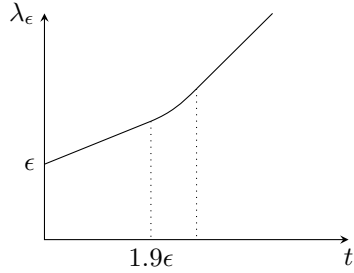
\begin{figure}[ht]
\begin{tikzpicture}
\draw [<->] (0,3) -- (0,0) -- (4,0);
\draw  (0,1) -- (1.4,1.57);
\draw  (2,2) -- (3,3);
\draw  (1.4,1.57) to [out=23,in=-135] (2,2);
\draw [dotted] (1.4,1.57)--(1.4,0);
\draw [dotted] (2,2)--(2,0);
\node [left] at (0,3) {$\lambda_{\epsilon}$};
\node [below] at (4,0) {$t$};
\node [below] at (1.4,0) {$1.9\epsilon$};
\node [left] at (0,1) {$\epsilon$};
\end{tikzpicture}
\caption{ The graph of $\lambda_{\epsilon}(t)$.}
\end{figure}

We define a self-embedding $p_{\epsilon}$ of $\tW$ as follows: on the complement of the $2\epsilon$-collar, the embedding $p_{\epsilon}$ is the identity, and on the $2\epsilon$-collar, it is defined to be
\[
p_{\epsilon}(t,x)= (\lambda_{\epsilon}(t),x).
\]
 \emph{The pushing collar map} is a group homomorphism 
\[
c_{\epsilon}: \dDiff(\tW,\partial)\rightarrow \dDiff(\tW,\partial)
\]
so that the image of $c_{\epsilon}$ lies in the subgroup  of $\dDiff(\tW,\partial)$ consisting of those diffeomorphisms whose supports are away from the $\epsilon$-collar. We denote this subgroup by $\dDiff(\tW, \text{rel $\epsilon$-collar})$. For every $\epsilon<1/2$, we define
\[
c_{\epsilon}(f)(x):= \begin{cases}
x & \text{if } x\in\epsilon{-collar}\\ p_{\epsilon}(f(p_{\epsilon}^{-1}(x))) &\text{if }x\in\tW\backslash\epsilon{-collar}.
\end{cases}
\]
\end{defn}
 \begin{lem}\label{technical}
The pushing collar map $c_{\epsilon}$ acts as the identity on $H_*(\dDiff(\ttW,\partial);\bZ)$.
\end{lem}
\begin{proof}
 In order to show that $c_{\epsilon}$ induces the identity on homology, we invoke a lemma from \cite[Lemma 3.7]{mcduff1980homology}. In this lemma, McDuff showed if $K$ is a discrete group and $c$ is an endomorphism of $K$ such that the restriction of $c$ to any finite subset of $K$ is  a conjugation map by some element of the group that may depend on the finite subset that $c$ is restricted to, then $c$ acts as the identity on the group homology. Hence, to prove $c_{\epsilon}$ acts as the identity on homology, it suffices to prove that for any finite set of elements $\{ f_1,f_2,\dots, f_n\}$ in $\dDiff(\tW,\partial)$, there is a group element $h$ such that for all $1\leq i\leq n$, we have $h(f_i)h^{-1}=c_{\epsilon}(f_i)$. We choose a  positive $\delta\leq \epsilon$ such that the $\delta$-collar is fixed by all $f_i$'s; such $\delta$ exists because every element in $\dDiff(\tW,\partial)$ is fixing a neighborhood of the boundary by definition. We define a diffeomorphism $h\in \dDiff(\tW,\partial)$ that maps the $\delta$-collar diffeomorphically to the $(\epsilon+\delta/2)$-collar and maps the complement of $\delta$-collar by the embedding $p_{\epsilon}$. Let $\gamma: [0,1]\to [0,1]$ be a diffeomorphism satisfying $\gamma(t)=\frac{2\epsilon+\delta}{2\delta}\cdot t$ for $t\leq 0.9\delta$ and $\gamma(t)=\lambda_{\epsilon}(t)$ for $t\geq \delta$.
 \begin{figure}[ht]
\begin{tikzpicture}
\draw [<->] (0,3) -- (0,0) -- (5,0);
\draw  (1,1) -- (2.4,1.57);
\draw  (3,2) -- (4,3);
\draw (0,0)--(0.7,0.8);
\draw (0.7,0.8) to [out=46, in =200] (1,1);
\draw  (2.4,1.57) to [out=24,in=-135] (3,2);
\draw [dotted] (0.7,0.8)--(0.7,0);
\draw [dotted] (1,1)--(0,1);
\node [left] at (0,3) {$\gamma(t)$};
\node [below] at (5,0) {$t$};
\node [below] at (0.7,0) {$0.9\delta$};
\node [left] at (0,1) {$\frac{2\epsilon+\delta}{2}$};
\end{tikzpicture}
\caption{ The graph of $\gamma(t)$.}
\end{figure}
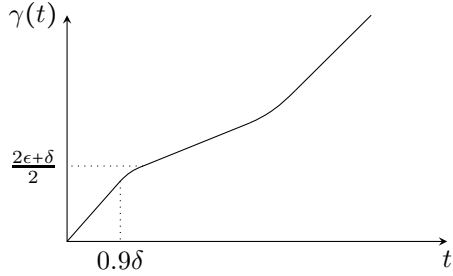
 
  We define the diffeomorphism $h$ to be
  \[
 h(x):=\begin{cases}
 (\gamma(t),y) & \text{if } x=(t,y)\text{ is in }\delta\text{-collar}\\p_{\epsilon}(x) & \text{otherwise.}
 \end{cases}
 \]
  It is straightforward to check that for the diffeomorphism $h$,  we have $h(f_i)h^{-1}=c_{\epsilon}(f_i)$ for $i=1,\dots,n$.
\end{proof}
\begin{cor}\label{collar}
The natural injection of $\beta: \dDiff(\ttW,\text{rel $\epsilon$-collar})\hookrightarrow \dDiff(\ttW,\partial)$ induces an isomorphism on homology.
\end{cor}
\begin{rem}
The same statement is also true for volume preserving diffeomorphisms \cite[Theorem 1]{MR699012}, but the proof in that case is much harder.
\end{rem}
\begin{proof} We have to apply \Cref{technical} twice as follows. Since the image of the homomorphism $c_{\epsilon}$ lies in $\dDiff(\tW,\text{rel $\epsilon$-collar})$, we write $c_{\epsilon}$ as a composition $\beta\circ c'_{\epsilon}$ where $c'_{\epsilon}$ is the same map as $c_{\epsilon}$ but with restricted codomain
  \[
  \begin{tikzpicture}[node distance=2.2cm, auto]
   \node (A) {$\dDiff(\tW,\partial)$};
   \node (B) [right of=A, node distance=3.2cm]{$\dDiff(\tW,\text{rel $\epsilon$-collar})$};
   \node (C) [right of=B, node distance=3.2cm]{$\dDiff(\tW,\partial).$};
   \draw[->] (A) to node {$c'_{\epsilon}$} (B);
   \draw[->] (B) to node {$\beta$} (C);
  \end{tikzpicture}
 \]
  \Cref{technical} implies that $\beta$ induces surjection on homology. To prove that it also induces injection on homology, we consider the composition  $c'_{\epsilon}\circ \beta$
   \[
  \begin{tikzpicture}[node distance=2.2cm, auto]
   \node (A) {$\dDiff(\tW,\text{rel $\epsilon$-collar})$};
   \node (B) [right of=A, node distance=3.2cm]{$\dDiff(\tW,\partial)$};
   \node (C) [right of=B, node distance=3.2cm]{$\dDiff(\tW,\text{rel $\epsilon$-collar}).$};
   \draw[->] (A) to node {$\beta$} (B);
   \draw[->] (B) to node {$c'_{\epsilon}$} (C);
  \end{tikzpicture}
 \] 
 Similar to \Cref{technical}, we can show $c'_{\epsilon}\circ \beta$ acts as the identity on homology of $\dDiff(\tW,\text{rel $\epsilon$-collar})$, hence $\beta$ induces also an injective map on homology.
\end{proof}

The difference between standard and nonstandard stabilization maps is that the standard stabilization maps  embed $\partial\W$ into   $\tW_{g+1,1}\backslash \partial$ but nonstandard stabilization maps send part of the boundary $\partial\W$ to the boundary of $\tW_{g+1,1}$. The method we use to prove homological stability for $\dDiff(\W,\partial)$ works better if we use nonstandard stabilization maps. The next theorem roughly says if we have homological stability induced by nonstandard stabilization maps, we also have homological stability induced by standard stabilization maps.

 \begin{figure}[ht]\label{fig2}
 \begin{tikzpicture}[scale=0.2]
  \begin{scope}[shift={(11,0)}]
\draw [line width=1.05pt] (-1.5,2.5)--(13,2.5);
\draw [line width=1.05pt] (13,-2.5) arc (-90:90:2.5);
\draw [line width=1.05pt] (-1.5,-2.5)--(13,-2.5);
\draw [line width=1.1pt] (4.6,.5) arc (240:300:2.5 and 6.75);
\draw   [line width=1.7pt] (6.87, 0.13) arc (48:150:1.2 and 0.9);
\begin{scope}[shift={(-3,0)}]
\draw [line width=1.05pt] (4.6,.5) arc (240:300:2.5 and 6.75);
\draw   [line width=1.7pt] (6.87, 0.13) arc (48:150:1.2 and 0.9);
\end{scope}
\begin{scope}[shift={(6,0)}]
\draw [line width=1.05pt] (4.6,.5) arc (240:300:2.5 and 6.75);
\draw   [line width=1.7pt] (6.87, 0.13) arc (48:150:1.2 and 0.9);
\end{scope}
\node at  (8,0) {$\bullet$};
\node at (9,0) {$\bullet$};
\node at (10,0) {$\bullet$};
\end{scope}
\begin{scope}[shift={(-16,0)}]
\draw [line width=1.05pt] [dashed] (-1.5,0) ellipse (0.5 and 2.5);
\draw[line width=1.05pt] [dashed] (0,.-2.5) arc (-90:90:0.5 and 2.5);
\draw [line width=1.05pt] (0,2.5) arc (90:270:0.5 and 2.5);
\draw [line width=1.05pt] (0,2.5)--(13,2.5);
\draw [line width=1.05pt][dashed] (-1.5,2.5)--(0,2.5);
\draw [line width=1.05pt] (13,-2.5) arc (-90:90:2.5);
\draw [line width=1.05pt] (0,-2.5)--(13,-2.5);
\draw [line width=1.05pt][dashed] (-1.5,-2.5)--(0,-2.5);
\draw [line width=1.1pt] (4.6,.5) arc (240:300:2.5 and 6.75);
\draw   [line width=1.7pt] (6.87, 0.13) arc (48:150:1.2 and 0.9);
\begin{scope}[shift={(-3,0)}]
\draw [line width=1.05pt] (4.6,.5) arc (240:300:2.5 and 6.75);
\draw   [line width=1.7pt] (6.87, 0.13) arc (48:150:1.2 and 0.9);
\end{scope}
\begin{scope}[shift={(6,0)}]
\draw [line width=1.05pt] (4.6,.5) arc (240:300:2.5 and 6.75);
\draw   [line width=1.7pt] (6.87, 0.13) arc (48:150:1.2 and 0.9);
\end{scope}
\node at  (8,0) {$\bullet$};
\node at (9,0) {$\bullet$};
\node at (10,0) {$\bullet$};
\end{scope}
\node at (-10,-4.1) {$\W\backslash\epsilon\text{-collar}$};
\node at (13,-4.1) {$\tW_{g+1,1}$};
\draw  [->, line width=1.05pt] (.5,0)--(4,0);
\node at (-24,-5.6) {$f:$};
\begin{scope}[shift={(-8.5,-9)}]
\draw [line width=1.05pt] (-0.5,2.5) arc (90:270:2.5);
\draw [line width=1.05pt] (3,2.5)--(5,2.5);
\draw [line width=1.05pt] (2.5,1.1)--(4.6,1.1);
\draw [line width=1.05pt] (-0.5,2.5)--(3,2.5);
\draw [line width=1.05pt] (-0.5,-2.5)--(3,-2.5);
\draw [line width=1.05pt] (5,2.5)--(6.5,2.5);
\draw [line width=1.05pt] (5,-2.5)--(6.5,-2.5);
\draw [line width=1.05pt] (6.5,0) ellipse (0.5 and 2.5);
\draw [line width=1.05pt] (5,2.5) arc (90:270:0.5 and 2.5);
\draw [line width=1.05pt][dashed] (5,2.5) arc (90:-90:0.5 and 2.5);
\draw [line width=1.05pt] (3,2.5) arc (90:270:0.5 and 2.5);
\draw [line width=1.05pt][dashed] (3,2.5) arc (90:23:0.5 and 2.5);
\draw [line width=1.05pt] (3,-2.5) arc (-90:20:0.5 and 2.5);
\begin{scope}[shift={(-5.6,0)}]
\draw [line width=1.05pt] (4.6,.5) arc (240:300:2.5 and 6.75);
\draw   [line width=1.7pt] (6.87, 0.13) arc (48:150:1.2 and 0.9);
\end{scope}
\draw  [->, line width=1.05pt] (8,0)--(12,0);
\node at (1.6,-4) { \small$\tW_{1,1}\natural\epsilon\text{-collar} $};
\end{scope}
\begin{scope}[shift={(3.2,-9)}]
 \draw [line width=1.05pt] (7.6,-2.5) to [out=145, in= 250] (6.5, 2.5);
 \draw [line width=1.05pt] (6.5,2.5) to [out=340, in= 70] (7.6, -2.5);
 \draw [line width=1.05pt] (2,2.5)--(6.5,2.5);
\draw [line width=1.05pt] (2,-2.5)--(7.6,-2.5);
\draw [line width=1.05pt] (2,2.5) arc (90:270:0.5 and 2.5);
\draw [line width=1.05pt][dashed] (2,2.5) arc (90:-90:0.5 and 2.5);
\begin{scope}[shift={(-1.7,0)}]
\draw [line width=1.05pt] (4.6,.5) arc (240:300:2.5 and 6.75);
\draw   [line width=1.7pt] (6.87, 0.13) arc (48:150:1.2 and 0.9);
\end{scope}
\end{scope}
\begin{scope}[shift={(3.2,0)}]
 \draw [line width=1.05pt][dashed] (7.6,-2.5) to [out=145, in= 250] (6.5, 2.5);
 \draw [line width=1.05pt][dashed] (6.5,2.5) to [out=340, in= 70] (7.6, -2.5);
 \draw [line width=1.05pt] (2,2.5)--(6.5,2.5);
\draw [line width=1.05pt] (2,-2.5)--(7.6,-2.5);
\draw [line width=1.05pt] (2,2.5) arc (90:270:0.5 and 2.5);
\draw [line width=1.05pt][dashed] (2,2.5) arc (90:-90:0.5 and 2.5);
\begin{scope}[shift={(-1.7,0)}]
\draw [line width=1.05pt] (4.6,.5) arc (240:300:2.5 and 6.75);
\draw   [line width=1.7pt] (6.87, 0.13) arc (48:150:1.2 and 0.9);
\end{scope}
\end{scope}

\end{tikzpicture}
\caption{The diffeomorphism $f$ maps $\tW_{1,1}\natural \epsilon\text{-collar}$ to a submanifold of $\tW_{g+1,1}$ which is diffeomorphic to $\tW_{1,2}$.}
\end{figure}
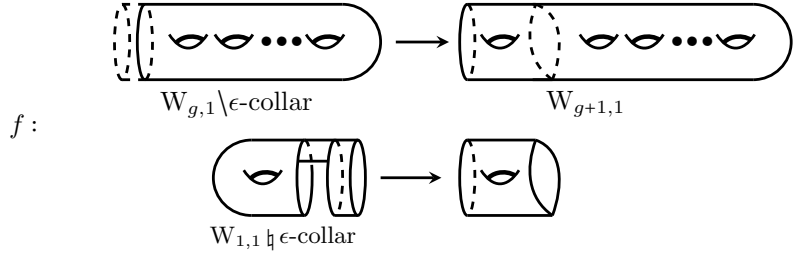

 \begin{thm}\label{StabilizationMap}
 Suppose for an integer $k$, the map induced on homology by a {\it nonstandard} stabilization map 
   \[
  \begin{tikzpicture}[node distance=3.9cm, auto]
   \node (A) {${ns_f}_*: H_k(\dDiff(\WW,\partial);\bZ)$};
   \node (B) [right of=A, node distance=4.9cm]{$H_k(\dDiff(\ttW_{g+1,1},\partial);\bZ)$};
   \draw[->] (A) to node {$$} (B);
  \end{tikzpicture}
 \]
is an isomorphism (a surjection), then the induced map on homology by any {\it standard} stabilization map 
   \[
  \begin{tikzpicture}[node distance=3.2cm, auto]
   \node (A) {${s_j}_*: H_k(\dDiff(\WW,\partial);\bZ)$};
   \node (B) [right of=A, node distance=4.9cm]{$H_k(\dDiff(\ttW_{g+1,1},\partial);\bZ)$};
   \draw[->] (A) to node {$$} (B);
  \end{tikzpicture}
 \]
 is also an isomorphism (a surjection). 
\end{thm}
\begin{proof} 
Since different standard stabilization maps are conjugate, we only need to prove the claim for one standard stabilization map. Let $f: \W\natural \tW_{1,1}\to \tW_{g+1,1}$ be the diffeomorphism to which the nonstandard stabilization map $ns_f$ is associated. We want to find an embedding $j: \W \hookrightarrow \tW_{g+1,1}\backslash \partial$ so that the associated standard stabilization map $s_j$ makes the following diagram commutative 
 \[
 \begin{tikzpicture}[node distance=2cm, auto]
  \node (A) {$\dDiff(\W,\partial)$};
  \node (B) [right of=A, below of=A, node distance=1.9cm]{$\dDiff(\W,\partial).$};
  \node (C) [right of=A, node distance=4cm]{$ \dDiff(\tW_{g+1,1},\partial)$};
  \draw [->] (A) to node {$s_j$} (C);
  \draw [->] (B) to node {$ns_f$} (C);
  \draw [->] (A) to node {$c_{\epsilon}$} (B);
 \end{tikzpicture}
\]
Note that $ns_f\circ c_{\epsilon}$ is given by the following composition
  \[
\dDiff(\W,\partial)\xrightarrow{c'_{\epsilon}}\dDiff(\W,\text{rel $\epsilon$-collar})\rightarrow \dDiff(\W\natural \tW_{1,1},\partial)\xrightarrow{\cong} \dDiff(\tW_{g+1,1},\partial)
 \]
 where the last isomorphism is induced by $f$. The image of $ns_f\circ c_{\epsilon}$ is supported away from $f(\tW_{1,1}\natural \partial\W\times [0,\epsilon])$ which is diffeomorphic to $\tW_{1,2}$ (see Figure 4). If we define $j$ to be $f\circ p_{\epsilon}$, then we have $s_j=ns_f\circ c_{\epsilon}$. By \Cref{technical}, the map $c_{\epsilon}$ induces the identity on homology, hence $s_j=ns_f\circ c_{\epsilon}$ induces an isomorphism (a surjection) on $k$-th homology if the same holds for $ns_f$.
 \end{proof}

\section{A simplicial resolution of $\BdDiff(\text{W}_{g,1},\partial)$}\label{resolution}

 In this section, we construct a semi-simplicial resolution for $\BdDiff(\text{W}_{g,1},\partial)$. To recall the  definition of semi-simplicial resolution, let $\Delta$ denote the category whose objects are non-empty totally ordered finite sets and whose morphisms are monotone maps. Let $\Delta_{\text{inj}}$ be the full subcategory of $\Delta$ with the same objects but only the injective maps as morphisms. A $\textit{semi-simplicial object}$ in a category $\mathcal{C}$ is a contravariant functor from $\Delta_{\text{inj}}$ to $\mathcal{C}$. For our purpose $\mathcal{C}$ is either the category of sets or the category of topological spaces. More concretely, we denote a semi-simplicial set (space) by $X_{\bullet} = \{ X_n |\  n=0,1,\dots\}$, which is a collection of sets (spaces) for each $n\geq 0$ and face maps $d_i:X_n\rightarrow X_{n-1}$ defined for $i=0,1,\dots, n$ that satisfy the usual identities $d_id_j=d_{j-1}d_i$ for $i<j$. Let $\Delta^n$ be the standard $n$-dimensional simplex and $d^i:\Delta^n\rightarrow \Delta^{n+1}$ be the inclusion of the $i$-th face where $i=0,\dots, n$. The geometric realization of a semi-simplicial set (space) $X_{\bullet}$ is 
  \[
|X_{\bullet}|  =\coprod_{n\geq 0}X_n\times \Delta^n/\sim
\]
where the equivalence relation is $(d_i(x),y)\sim(x,d^i(y))$. An augmented semi-simplicial space $X_{\bullet}\rightarrow X_{-1}$ is a semi-simplicial space $X_{\bullet}$ with a map $\epsilon:X_0\rightarrow X_{-1}$ called the augmentation map, which coequalizes  the face maps $d_0: X_1\rightarrow X_{0}$ and $d_1: X_1\rightarrow X_{0}$. The augmentation induces a map $|X_{\bullet}|\rightarrow X_{-1}$(see \cite[Section 2]{randal2009resolutions} for more details). In this section, we describe an augmented semi-simplicial space $X_{\bullet}\rightarrow  \BdDiff(\text{W}_{g,1},\partial)$ such that  the map induced by the augmentation $|X_{\bullet}|\rightarrow \BdDiff(\text{W}_{g,1},\partial)$ is $\left \lfloor (g-3)/{2}\right \rfloor $-connected.

  Let $H$ be the closure of the complement of $\tW_{g,1}$ in $\W\natural \tW_{1,1}$. We want to fix a submanifold  $[0,1]\times D^{2n-1} $ in $H$ that is used in the boundary connected sum. Note that $H$ can be  obtained from $\tW_{1,1}$ by gluing $[0,1]\times D^{2n-1}$ onto $\partial \tW_{1,1}$ along an orientation preserving embedding
 \[ e:\{ 1\} \times D^{2n-1} \rightarrow \partial \text{W}_{1,1} \]
 that we  choose once and for all. The point of  working with $H$ is, although it is diffeomorphic to $\text{W}_{1,1}$ after smoothing corners, it has a standard embedded  $[0,1]\times D^{2n-1} \subset H$. Assume that $( S^n, *)$ is a sphere with a chosen base point, then $ S^n\vee S^n=(S^n \times \{*\}) \cup(\{*\} \times S^n) \subset S^n\times S^n$. The manifold $\tW_{1,1}$ is the result of removing an open ball from $S^n\times S^n$. We assume this open ball is chosen so that $S^n\vee S^n$ is contained in $\text{int}(\text{W}_{1,1})$.
  \begin{defn}
  Let $\gamma$ be a path in $\text{int}(H)$ from $(0,0)\in [0,1]\times D^{2n-1}$ to some chosen point on $S^n\vee S^n$ that is not $(*,*)\in S^n\times S^n$ such that the interior of $\gamma$ in $H$ does not intersect  $S^n\vee S^n$ and the image of $\gamma$ agrees with $[0,1]\times \{0\}$ inside $[0,1]\times D^{2n-1}$. We define the \textit{core} $C\subset H$ to be:
  \[  C=(S^n\vee S^n)\cup \gamma([0,1])\subset H \]
 The core in $H$ is depicted in \Cref{core}.

  \end{defn}

  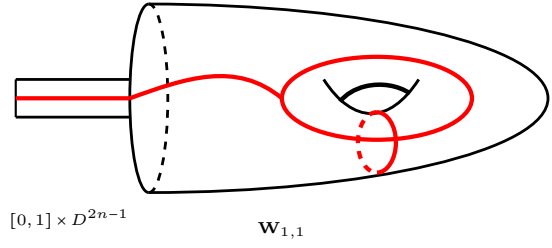
\begin{figure}[ht]
 \
\begin{tikzpicture}[scale=.5]

\begin{scope}[shift={(6,0)}]
\draw[line width=1.05pt] [dashed] (0,.-2.5) arc (-90:90:0.5 and 2.5);
\draw [line width=1.05pt] (0,2.5) arc (90:270:0.5 and 2.5);
\draw [line width=1.05pt] (0,0)+(0,-2.5) arc (-90:90:10.5 and 2.5);
\draw [line width=1.1pt] (4.6,.5) arc (240:300:2.5 and 6.75);
\draw   [line width=1.7pt] (6.87, 0.13) arc (48:150:1.2 and 0.9);
\draw  [red, ultra thick](6,-0.37) arc (90:-90: 0.5 and .8);
\draw  [red, ultra thick, dashed](6,-0.37) arc (90:270: 0.5 and .8);
\draw  [red, ultra thick] (6,0) ellipse (2.5 and 1.1);
\draw  [red, ultra thick] (-0.5,0) to [out= 20, in=140] (3.5,0);
\draw [line width=1.05pt] (-0.48,0.5)--(-3.5,0.5);
\draw [line width=1.05pt] (-0.48,-0.5)--(-3.5,-0.5);
\draw [ thick] (-3.5,-0.5)--(-3.5,0.5);
\draw [red,ultra thick] (-0.5,0)--(-3.5,0);
\node at (-2.1,-3.2) {\tiny\bf{$[0,1]\times D^{2n-1}$}};
\node at (3.5,-3.5) {\tiny\bf{$\text{W}_{1,1}$}};
\end{scope}
 
\end{tikzpicture}  
  \caption{ \textcolor{red}{Core} in $H$}
    \label{core}

 \end{figure}

 We make our choice of $\gamma$ once and for all to have a fixed  \textit{core} $C$ in $H$. Note that $H$ is homotopy equivalent to its core. We want to consider those embeddings of  $H$ into a manifold $\tW$ that map $\{ 0\}\times D^{2n-1}$  to $\partial \tW$ and send the rest of $H$ to the interior of $\tW$.  Similarly an embedding of a neighborhood $U$ of the core $C$ into $\tW$ has to send $U\cap\{ 0\}\times D^{2n-1}$ to $\partial \tW$ and send the rest of $U$ to the interior of $\tW$. 
 \begin{defn} A germ of an embedding of the core $C\subset H$ into $ \text{W}$ is an equivalence class of the pair, $[\phi]:=(U,\phi)$, where $U$ is a neighborhood of the core $C\subset H$ and $ \phi: U\hookrightarrow \text{W}$ is a smooth embedding of $U$ into  $\text{W}$. We say $(U,\phi)$ is equivalent to $(U',\phi')$ if and only if  there exists an open neighborhood of the core $U''\subset U\cap U'$, such that $\phi|_{U''}=\phi'|_{U''}$.
\end{defn}
To define the semisimplicial resolution for $\BdDiff(\tW,\partial)$, we first use embeddings of the core $C$ into $\tW$ to define a simplicial complex on which $\dDiff(\tW,\partial)$ acts. Since diffeomorphisms in $\dDiff(\tW,\partial)$ are identity near the boundary, we only consider those embeddings of the core into $\text{W}$ that are ``horizontal"  near the boundary (see \Cref{collared} bellow). Furthermore, since nonstandard stabilization maps are defined by  the boundary connected sum, we want to fix part of the boundary that is not affected by the boundary connect sum. Hence we fix the data of an embedding $c:\bH=\bR_+\times\bR^{2n-1}\hookrightarrow \text{W}$ where $c^{-1}(\partial \tW)=\partial\bH$, even if we don't write $c$ along with $\tW$, we assume that we have  chosen the chart $c$ once and for all  to define the resolution for $\BdDiff(\text{W},\partial)$. 

Let $(\text{W},c)$ be a pair where $\text{W}$ is a manifold with a nonempty boundary and $c$ is the fixed chart on the boundary.  Let $\text{B}_{r}(0)$ denote the open ball of radius $r$ around the origin in $\bR^{2n-1}$. For any neighborhood $U$ of $C\subset H$,  there exits a small enough $r$ such that $[0,1]\times \text{B}_{r}(0)\subset U$, where $[0,1]\times \text{B}_{r}(0)$ is contained in the standard embedded  $[0,1]\times D^{2n-1} \subset H$.

  \begin{defn}\label{collared}
   A germ $ [\phi]$ of an embedding of the core into $\text{W}$, is called \emph{collared} embedding if for some $\epsilon,\eta>0$ and $(U,\phi)$ a representative of the germ $[\phi]$,  the restriction of $\phi$ to $[0,1]\times \text{B}_{\eta}(0)\subset U$  satisfies
\[ c^{-1}(\phi(x,p))=(x,p+t_{\phi})\]
  for all $p\in \text{B}_{\eta}(0)$ and $x<\epsilon$. Here, $t_{\phi}$ is the point $c^{-1}(\phi(0,0))\in \{0 \}\times \bR^{2n-1}$.
 \end{defn}
  
\begin{defn}  For a pair $(\text{W},c)$, we define a simplicial complex $\text{E}(\text{W})$ whose $p$ simplices consists of unordered set of \emph{collared} embeddings $\{[\phi_0],  [\phi_1],\dots,[ \phi_p]\}$ satisfying $\phi_i(C)\cap \phi_j(C)=\varnothing$  for all distinct $i,j$. We associate a semisimplicial set denoted by $\text{E}_{\bullet}(\text{W})$ to $\text{E}(\text{W})$ by  imposing all orderings on the simplices i.e. $\text{E}_{p}(\text{W})$ consists of $(p+1)$-tuples $([\phi_0],  [\phi_1],\dots,[ \phi_p])$ satisfying $\phi_i(C)\cap \phi_j(C)=\varnothing$  for all distinct $i,j$.
\end{defn}
    We want to  prove that  $|\tE_{\bullet}(\W)|$ is highly connected. There is a natural surjective map $\theta:  |\tE_{\bullet}(\W)|\to |\tE(\W)|$ and any choice of total order on vertices of $\tE(\W)$ gives a section $|\tE(\W)|\to |\tE_{\bullet}(\W)|$.  Essentially with the same argument as \cite[Lemma 5.5]{galatius2014homological}, one can  show that $|\tE(\W)|$ is highly-connected. We use the connectivity of this complex to show that $|\tE_{\bullet}(\W)|$ is highly connected. First we need to recall an important technique known as a ``coloring lemma". 
 \begin{defn}  A simplicial complex $K$ is called $ \it{weakly}$  $\it{Cohen-Macaulay}$ of dimension $n$ and it is denoted by  $\it{w}CM(K)\geq n$ if it is $(n-1)$-connected and the link of any $p$-simplex is $(n-p-2)$-connected. 
 \end{defn}
\begin{defn}
 We say that a simplicial map $f: X\rightarrow Y$ of simplicial complexes is simplexwise injective if its restriction to each simplex of $X$ is injective, i.e. the image of any $p$-simplex of $X$ is a non-degenerate $p$-simplex of $Y$.
\end{defn}
The following theorem {\it generalized coloring lemma} is proved in \cite[Theorem 2.4]{galatius2014homological}.
\begin{thm}\label{coloringlemma}
 Let $X$ be a simplicial complex and $f:\partial I^n\rightarrow |X|$ be a map which is simplicial with respect to some PL triangulation on $\partial I^n$. Then, if $\it{w}CM(X)\geq n$, the triangulation extends to a PL triangulation of $I^n$, and $f$ extends to a simplicial map $g:I^n\rightarrow |X|$ with the property that $g(\text{\textnormal{Link}}(v))\subset \text{\textnormal{Link}}(g(v))$ for each interior vertex $v\in \text{\textnormal{int}}(I^n)$. In particular, g is simplexwise injective if $f$ is.
\end{thm}

\begin{thm}\label{lcm}
 The simplicial complex $\text{\textnormal{E}}(\WW)$ is weakly Cohen-Macaulay with $wCM(\text{\textnormal{E}}(\WW))\geq (g-2)/2$. In particular, $|\text{\textnormal{E}}(\WW)|$ is at least $(g-4)/2$-connected.
\end{thm}
\begin{proof}
Similar argument to \cite[Lemma 5.5]{galatius2014homological} implies that $|\tE(\W)|$ is $(g-4)/2$-connected. Being weakly Cohen-Macaulay is the consequence of the cancellation theorem \cite[Corollary 6.3]{galatius2014homological}.
\end{proof}
\begin{thm}\label{connectivity}
The geometric realization $| \text{\textnormal{E}}_{\bullet}(\text{\textnormal{W}}_{g,1})|$ is $\left \lfloor (g-4)/{2} \right\rfloor $-connected. 
\end{thm}
\begin{proof}
For $k\leq (g-4)/2$, let $f:\partial I^{k+1}\to | \text{\textnormal{E}}_{\bullet}(\text{\textnormal{W}}_{g,1})|$ be a map. We can assume that $f$ is a simplicial map with respect to some PL triangulation on $\partial I^{k+1}$. \Cref{lcm} implies that $\partial I^{k+1}\xrightarrow{f} | \text{\textnormal{E}}_{\bullet}(\text{\textnormal{W}}_{g,1})|\to |\text{\textnormal{E}}(\WW)|$ is null-homotopic. To prove the claim, we need to solve the following homotopy lifting problem
\[ 
\begin{tikzpicture}[node distance=2cm, auto]
  \node (A) {$\partial I^{k+1}$};
  \node (B) [right of=A] {$|\tE_{\bullet}(\tW_{g,1})|$};
  \node (C) [below of=A, node distance=1.7cm] {$ I^{k+1}$};  
  \node (D) [below of=B, node distance=1.7cm] {$|\text{\textnormal{E}}(\WW)|.$};
  \draw[->, dashed] (C) to node  {$\tilde{h}$} (B);
  \draw[->] (C) to node {$h$} (D);
  \draw [right hook->] (A) to node {}(C);
  \draw [->] (A) to node {$f$} (B);
  \draw [->] (B) to node {$\theta$} (D);
\end{tikzpicture}
\]
Since $|\text{\textnormal{E}}(\WW)|$ is weakly Cohen-Macaulay with $wCM(\text{\textnormal{E}}(\WW))\geq (g-2)/2$, by \Cref{coloringlemma}, we can extend the triangulation of $\partial I^{k+1}$ to a triangulation of $I^{k+1}$ and change $h$ by a homotopy relative to $\partial I^{k+1}$ so that $h$ is PL and simplexwise injective on the interior of $I^{k+1}$. Given that $\theta$ has a section and $h$ is injective on the interior vertices, we can find a lift $\tilde{h}$. 
\end{proof}

 \subsection{Orbits of the action of $\dDiff(\W,\partial)$ on $\tE_\bullet(\W)$}
 In proving homological stability, it is convenient to have a transitive action. As we shall see, the action of $\dDiff(\W,\partial)$ on $\tE_\bullet(\W)$ is not transitive, but it turns out that the set of  orbits is  independent of $g$ (See \Cref{independenceofg}). To describe the set of the orbits, let $\text{C}_p(\bR^{2n-1})$  be the configuration space of ordered $p$ points in $\bR^{2n-1}$. To every $p$-simplex $\phi=([\phi_0],  [\phi_1],\dots,[ \phi_p])$ in $\tE_p(\W)$, we associate  $(t_{\phi_0}, t_{\phi_1},\dots,t_{\phi_p})\in\text{C}_{p+1}(\bR^{2n-1})$. The diffeomorphism group  $\dDiff(\W,\partial)$  fixes the boundary, hence it does not change $(t_{\phi_0}, t_{\phi_1},\dots,t_{\phi_p})$ associated to $\phi$ and if two $p$-simplices give two different elements in $\text{C}_{p+1}(\bR^{2n-1})$, they are obviously in different orbits. Therefore we can define a  map
  \[
  \begin{tikzpicture}[node distance=3.9cm, auto]
  \node (A) {$I: \tE_\bullet(\W)/\dDiff(\W,\partial)$};
  \node (B) [right of=A] {$\text{C}_{p+1}(\bR^{2n-1}).$};
  \draw [->] (A) to node {} (B);
\end{tikzpicture}
\]
We use $\text{C}_{p+1}(\bR^{2n-1})$ as an indexing set to describe the orbit decomposition of the action of $\dDiff(\W,\partial)$ on $\tE_p(\W)$. For every $\sigma \in \text{C}_{p+1}(\bR^{2n-1})$, we choose a fixed $p$-simplex $\phi_\sigma$ in $\tE_p(W)$ whose image under $I$ is $\sigma$. We denote the stabilizer of $\phi_\sigma$ by $\dDiff(\tW_{g,1},\partial)_\sigma$.
\begin{rem}\label{Kreck}
By Kreck's cancellation theorem \cite[Theorem D]{kreck1999surgery} or \cite[Corollary 4.5]{Galatius-Randal-Williams}, the complement of an open neighborhood of $(p+1)$ embedded cores given by $\phi_\sigma$ is diffeomorphic to $\tW_{g-p-1,1}$. Hence different choices for $\phi_{\sigma}$ lead to isomorphic stabilizer subgroups of $\dDiff(\WW,\partial)$. 
\end{rem}
  \begin{lem}\label{orbit}
   Every two $p$-simplices with the same image in $\text{\textnormal{C}}_{p+1}(\bR^{2n-1})$ are in the same orbit as $g-p\geq 2$.  Therefore, the orbit decomposition is 
   \[\text{\textnormal{E}}_p(\WW)=\coprod_{\sigma\in \text{\textnormal{C}}_{p+1}(\bR^{2n-1})}\dDiff(\WW,\partial)/{\dDiff(\ttW_{g,1},\partial)_{\sigma}}\]
  \end{lem}
\begin{proof}
 First, we prove the claim for $0$-simplices, so assume that $p=0$ and let $[\phi_1]$ and $[\phi_2]$ be $0$-simplices such that $I(\phi_1)=I(\phi_2)$. There exists a tubular neighborhood of the core $ U\supset C$ and positive numbers $\epsilon, \eta$  and representatives  $\phi_1$ and $\phi_2$ for the  germs of collared embedding of the core such that $[0,\epsilon)\times\text{B}_{\eta}(0)\subset U$ and $\phi_1([0,\epsilon)\times \text{B}_{\eta}(0))=\phi_2([0,\epsilon)\times \text{B}_{\eta}(0))$. 
 
 Suppose that we fix a collar neighborhood $\partial\W\times [0,1)\hookrightarrow \W$ and for a positive $\epsilon$, let $\W^{\epsilon}$ denote the complement of the $\epsilon$-collar. The intersections of $\phi_i(U)$'s and $\W^{2\epsilon/3}$ are  collared embeddings of the core in $\W^{2\epsilon/3}$. These intersections provide us with embeddings $\psi_i:H\hookrightarrow\W^{2\epsilon/3}$ satisfying $\psi_i([0,\epsilon/3)\times D^{2n-1})=\phi_i([2\epsilon/3,\epsilon)\times \text{B}_{\eta}(0))$. Using \cite[Corollary 4.4]{Galatius-Randal-Williams}, we can find a diffeomorphism $l\in \mathrm{Diff}^{\delta}(\W^{2\epsilon/3})$ that is isotopic to the identity on the boundary $\partial\W\times \{2\epsilon/3\}$ such that $l\circ\psi_1=\psi_2$. Furthermore, from the proof of \cite[Corollary 4.4]{Galatius-Randal-Williams}, it is clear that $l$ can be chosen so that the restriction of $l$ to $\partial\WW\times [2\epsilon/3,\epsilon)$ be equal to $l|_{\partial \W^{2\epsilon/3}}\times \text{id}$.  Since $l|_{\partial \W^{2\epsilon/3}}$ is isotopic to the identity, we have a map 
 \[
 h:\partial \W^{2\epsilon/3}\times [\epsilon/3, 2\epsilon/3]\longrightarrow \partial \W^{2\epsilon/3}
 \]
where for all $x\in \partial \W^{2\epsilon/3}$, $h$ satisfies $h(x,t)= x$ for $t\in [\epsilon/3,4\epsilon/9)$ and $h(x,t)=l(x)$ for $t\in(5\epsilon/9,2\epsilon/3]$.  Hence, by gluing $l$, $h$ and the identity on $\epsilon/3$-collar together, we obtain a diffeomorphism $f\in \dDiff(\W,\partial)$ that sends  $\phi_1(V)$ to  $\phi_2(V)$ for an open neighborhood $V$ containing the core  such that $C\subset V\subset U$.

Now let $\phi_i=([\phi^0_i], [\phi^1_i],\dots, [\phi^p_i])$ for $i=1,2$ be two $p$-simplices so that $I(\phi_1)=I(\phi_2)$. There exists  a diffeomorphism $f_0\in \dDiff(\W,\partial)$ such that for a neighborhood of the core $C\subset U_0$, $f_0$ satisfies $f_0\circ\phi^0_1(U_0)=\phi^0_2(U_0)$. We can choose $U_0$ so that $\W\backslash \phi^0_i(U_0)$ is a manifold with boundary (without corners). By the cancellation theorem  \cite[Corollary 4.5]{Galatius-Randal-Williams} the manifold $\W\backslash \phi^0_2(U_0)$ is diffeomorphic to $\tW_{g-1,1}$. Note that since $f_0\circ \phi^1_1(C),\phi^1_2(C)\subset \W\backslash  \phi^0_2(U_0)$, by the same argument as above,  a diffeomorphism $f_1\in \dDiff(\W\backslash  \phi^0_2(U_0),\partial)$ exists such that $f_1\circ\phi^1_1(U_1)=\phi^1_2(U_1)$ for a neighborhood $U_1$ of the core. We extend $f_1$ via the identity to a  diffeomorphism of $\W$. By repeating this argument, we obtain a diffeomorphism $f=f_p\circ f_{p-1}\circ\cdot\circ f_0$ that sends the first $p$-simplex to the other.
\end{proof}

%
Fix once for all a coordinate patch near the boundary $c:\bH=\bR_+\times\bR^{2n-1}\hookrightarrow \W$, which is disjoint from the embedding $e:\{0\}\times D^{2n-1}\hookrightarrow \partial\W\subset \W$ that we used to define the boundary connected sum $\W\natural \tW_{1,1}$ and the nonstandard stabilization map. 
\begin{cor}\label{independenceofg}
For the pair $(\WW,c)$, the map $\text{\textnormal{E}}_{\bullet}(\WW)\rightarrow \text{\textnormal{E}}_{\bullet}(\ttW_{g+1,1})$ induced by the nonstandard stabilization map is a bijection between orbits of the action of  $\dDiff(\WW,\partial)$ on $\text{\textnormal{E}}_{\bullet}(\WW)$ and the action of $\dDiff(\ttW_{g+1,1},\partial)$ on $\text{\textnormal{E}}_{\bullet}(\ttW_{g+1,1})$.
\end{cor}
We now use the high connectivity of the realization of $\tE_{\bullet}(\W)$ to construct a semisimplicial resolution for $\BdDiff(\W,\partial)$, meaning a semisimplicial space $X_{\bullet}$ with an augmentation map $X_{\bullet}\rightarrow \BdDiff(\W,\partial)$ such that the map $|X_{\bullet}|\rightarrow \BdDiff(\W,\partial)$ is highly connected.
\begin{con}\label{construction}
Recall from \Cref{collared} that for the pair $(\W,c)$, we defined $\tE_p(\W)$ to be the set of $p$-tuples of disjoint \emph{collared} embedded cores . Let
 \[ X_p= (\tE\dDiff(\W,\partial)\times \tE_p(\W))/ {\dDiff(\W,\partial)}\]
\end{con}
\noindent be the homotopy quotient of the action of $\dDiff(\W,\partial)$ on $\tE_p(\W))$. The face maps of $X_{\bullet}$ are induced by the face maps of the semisimplicial set $ \tE_{\bullet}(\W)$. Note that $X_{\bullet}$ is a semisimplicial space augmented over $\BdDiff(\W,\partial)$.
 \begin{prop}
  $X_{\bullet}$ is a $\left \lfloor (g-2)/2 \right\rfloor $-resolution for $\BdDiff(\WW,\partial)$, i.e. the map $|X_{\bullet}|\rightarrow \BdDiff(\WW,\partial)$ induced by the augmentation is $\left \lfloor (g-2)/{2}\right \rfloor $-connected. 
 \end{prop}
 \begin{proof}
 For an augmented semisimplicial space $\epsilon_{\bullet}: X_{\bullet}\rightarrow X_{-1}$, the homotopy fiber of $|X_{\bullet}|\rightarrow X_{-1}$ can be computed levelwise \cite[Lemma 2.1.]{randal2009resolutions}, meaning that the following square is weakly homotopy pullback square
 \[  
 \begin{tikzpicture}[node distance=2cm, auto]
  \node (A) {$|\text{hofib}( \epsilon_{\bullet})|$};
  \node (B) [right of=A] {$|X_{\bullet}|$};
  \node (C) [below of=A, node distance=1.7cm] {$*$};  
  \node (D) [below of=B, node distance=1.7cm] {$X_{-1}.$};
  \draw[->] (C) to node  {}(D);
  \draw [->] (A) to node {}(C);
  \draw [->] (A) to node {$$} (B);
  \draw [->] (B) to node {$$} (D);
\end{tikzpicture}
\]
Therefore, from the construction above, we obtain the semi-simplicial space $X_{\bullet}$ augmented over $\BdDiff(\W,\partial)$. The levelwise fiber of the augmentation map $X_{\bullet}\rightarrow \BdDiff(\W,\partial)$ is the semisimplicial space $\tE_{\bullet}(\W)$ whose geometric realization $|\tE_{\bullet}(\W)|$  by   \Cref{connectivity} is $\left \lfloor (g-4)/{2}\right \rfloor -\text{connected}$. Hence, $X_{\bullet}$ is a $\left \lfloor (g-2)/2 \right\rfloor $-resolution for $\BdDiff(\W,\partial)$.
 \end{proof}

\section{Proof of Theorem \ref{MainTheorem}}\label{spectralsequence}
 \begin{proof} We use the ``relative" spectral sequence argument in the sense of \cite[Proposition 4.2]{charney1987generalization} to prove homological stability by induction. The nonstandard stabilization map induces a semisimplicial map
 \[
{ns_f}_{\bullet}: X_{\bullet}(\W)\to X_{\bullet}(\tW_{g+1,1}).
 \]
The relative spectral sequence (see \cite[Section 2.1]{randal2009resolutions}) for this semisimplicial map  takes the form
\[ 
 \begin{tikzpicture}[node distance=6cm, auto]
  \node (A) {$E^1_{p,q}=H_q(X_p(\tW_{g+1,1}), X_p(\W))$};
  \node (B) [right of=A] {$H_{p+q}(|X_p(\tW_{g+1,1})|, |X_p(\W)|).$};
  \draw [->,double] (A) to node {$$} (B);
\end{tikzpicture}
\]

The facts that $X_{\bullet}(\tW_{g+1,1})$ is a $(g-1)/2$-resolution for $\BdDiff(\tW_{g+1,1},\partial)$ and $X_{\bullet}(\W)$ is a $(g-2)/2$-resolution for $\BdDiff(\tW_{g,1},\partial)$ imply the map  
\[
H_{p+q}(|X_p(\tW_{g+1,1})|, |X_p(\W)|)\to H_{p+q}(\BdDiff(\tW_{g+1,1},\partial), \BdDiff(\tW_{g,1},\partial)),
\]
is surjective as long as $p+q\leq (g-2)/2$ and isomorphism for $p+q< (g-2)/2$ . In order to prove \Cref{MainTheorem}, we need to show that \[H_{k}(\BdDiff(\tW_{g+1,1},\partial), \BdDiff(\tW_{g,1},\partial))=0\] as long as $k\leq (g-2)/2$.

{\it Claim:} For $p\geq 0$ and $q\leq (g-p-3)/2$, we have 
\[ E^1_{p,q}=\bigoplus_{\sigma} H_q(\BdDiff(\tW_{g+1,1},\partial)_{\sigma}, \BdDiff(\tW_{g,1},\partial)_{\sigma})=0. \]
Note that by \Cref{Kreck} each summand in $E^1_{p,q}$ is in fact isomorphic to $$H_q(\BdDiff(\tW_{g-p,1},\partial), \BdDiff(\tW_{g-p-1,1},\partial)).$$ Thus for $p\geq 0$ and $q\leq (g-p-3)/2$, the above group   is zero by induction on $g$.

 The first page of the spectral sequence in the range that we are interested in looks like
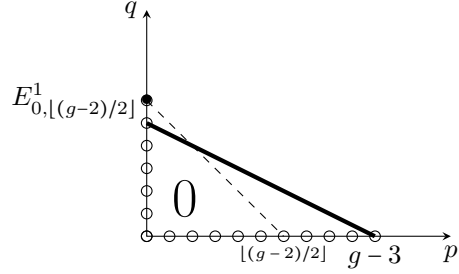
\begin{figure}[ht]\label{ss}
\begin{tikzpicture}
\draw [<->] (0,3) -- (0,0) -- (4,0);
\draw [ultra thick] (0,1.5) -- (3,0);
\draw [dashed] (0,1.8) -- (1.8,0);
\node [below] at (3,0) {$g-3$};
\node [left] at (0,1.8) {$E^1_{0,\left\lfloor(g-2)/2\right\rfloor}$};
\node at (0,1.8) {$\bullet$};
\node at (0.5,0.5) {\Huge$0$};
\node [left] at (0,3) {$q$};
\node [below] at (4,0) {$p$};
\node [below] at (1.8,0) {\tiny$\left\lfloor(g-2)/2\right\rfloor$};
\foreach \x in {0,...,6}
\draw (0,\x/3.33) circle (0.07cm);
\foreach \x in {0,...,10}
\draw (\x/3.33,0) circle (0.07cm);
\end{tikzpicture}
\caption{ First page of spectral sequence}
\end{figure}

On the first page of the spectral sequence, everything below the thick line that is given by $p+2q=g-3$ in Figure $6$, is zero  by induction. If $g$ is an odd number, then $E^1_{p,q}=0$ for $p+q\leq \left\lfloor(g-2)/2\right\rfloor$. If $g$ is even, then everything except $E^1_{0,\left\lfloor(g-2)/2\right\rfloor}$ on the dashed line in Figure $6$ and below is zero . In order to finish the proof, we need to show that the image of $E^1_{0,\left\lfloor(g-2)/2\right\rfloor}$ in $E^\infty_{0,\left\lfloor(g-2)/2\right\rfloor}$ vanishes.
Given \Cref{relative} below, all cycles in \[E^1_{0,\left\lfloor(g-2)/2\right\rfloor}= \bigoplus_{\sigma \in \text{C}_{1}(\bR^{2n-1})} H_{\left\lfloor(g-2)/2\right\rfloor}(\BdDiff(\tW_{g+1,1},\partial)_{\sigma},\BdDiff(\tW_{g,1},\partial)_{\sigma})\] 
die in $E^{\infty}$-page  which completes the proof of \Cref{MainTheorem}.
  \end{proof}
  
  To prove that cycles in $E^1_{0,\left\lfloor(g-2)/2\right\rfloor}$ do not survive to the $E^{\infty}$-page, we modify a factorization theorem  that is formulated for spaces in \cite[Prop 6.7]{randal2009resolutions} and for groups in \cite[Lemma 2.5]{wahl2010homological}. Let us first recall its statement for groups. Let $G_0, G_1, G_2$ be subgroups of a group $G$ fitting into a  diagram 
 \[
 \begin{tikzpicture}[node distance=1.5cm, auto]
  \node (A) {$G_0$};
  \node (B) [right of=A] {$G_1$};
  \node (C) [below of=A] {$G_2$};  
  \node (D) [below of=B] {$G$};
  \draw[right hook->] (C) to node  {}(D);
  \draw [right hook->] (A) to node {}(C);
  \draw [right hook->] (A) to node {$$} (B);
  \draw [right hook->] (B) to node {$$} (D);
  \draw [->, dotted] (B) to node [swap]{$t$} (C);
\end{tikzpicture}
\]
where there exists $t\in G$ such that $G_1\subset t\cdot G_2\cdot t^{-1}$ and $G_0$ is fixed under the conjugation by the element $t$, then the map $H_k(G_1,G_0)\rightarrow H_k(G,G_2)$ canonically factors through $H_{k-1}(G_0)$. To prove that every element in $E^1_{0,\left\lfloor(g-2)/2\right\rfloor}$ dies in the $E^{\infty}$-page, we invoke Randal-Williams' idea to show that for every $\sigma \in \text{C}_{1}(\bR^{2n-1})$, the image of any cycle 
\[[x] \in H_{\left\lfloor(g-2)/2\right\rfloor}(\BdDiff(\tW_{g+1,1},\partial)_{\sigma},\BdDiff(\tW_{g,1},\partial)_{\sigma})\]
under the natural inclusion  into {\small$H_{\left\lfloor(g-2)/2\right\rfloor}(\BdDiff(\tW_{g+1,1},\partial),\BdDiff(\tW_{g,1},\partial))$} factors (non-canonically) through {\small$H_{\left\lfloor(g-4)/2\right\rfloor}(\dDiff(\tW_{g,1},\partial)_{\sigma})$}.

 To define the nonstandard stabilization map, recall that we fixed an embedding $e:\{0\}\times D^{2n-1}\hookrightarrow \partial\W$ which does not intersect our chosen chart $c:\{0\}\times\bR^{2n-1}\hookrightarrow \partial\W$, then we attached $H$ to $\W$ via gluing $\{0\}\times D^{2n-1}\subset H$ to $e(\{0\}\times D^{2n-1})$ to obtain $\W \natural \tW_{1,1}$.  We are interested in the following diagram where the horizontal maps are given by the nonstandard stabilization map and the vertical maps are inclusions of the stabilizer subgroups
\[
 \begin{tikzpicture}[node distance=1.5cm, auto]
  \node (A) {$\dDiff(\ttW_{g,1},\partial)_{\sigma}$};
  \node (B) [right of=A, node distance=3.5cm] {$\dDiff(\ttW_{g+1,1},\partial)_{\sigma}$};
  \node (C) [below of=A] {$\dDiff(\ttW_{g,1},\partial)$};  
  \node (D) [below of=B] {$\dDiff(\ttW_{g+1,1},\partial).$};
  \draw[right hook->] (C) to node  {}(D);
  \draw [right hook->] (A) to node {}(C);
  \draw [right hook->] (A) to node {$$} (B);
  \draw [right hook->] (B) to node {$$} (D);
\end{tikzpicture}
\]

 The caveat is the group theoretic lemma  \cite[Lemma 2.5]{wahl2010homological} does not quite apply, because there is no $t\in \dDiff(\tW_{g+1,1},\partial)$ that conjugates $\dDiff(\tW_{g+1,1},\partial)_{\sigma}$ into $\dDiff(\tW_{g,1},\partial)$ fixing $\dDiff(\tW_{g,1},\partial)_{\sigma}$. It turns out, however, that for every finite set of diffeomorphisms $S=\{f_1,\dots,f_n\}$ stabilizing $\sigma$, there exits $t_S\in \dDiff(\tW_{g+1,1},\partial)$ depending on $S$ that conjugates every element of $S$ into $\dDiff(\tW_{g,1},\partial)$. As we shall see, existence of $t_S$ for every finite set $S$ of diffeomorphisms is enough to prove that all cycles in $E^1_{0,\left\lfloor(g-2)/2\right\rfloor}$ die in the $E^{\infty}$-page.
\begin{lem}\label{relative}
For every $\sigma \in \text{\textnormal{C}}_{1}(\bR^{2n-1})$, the following commutative diagram
 \[  
 \begin{tikzpicture}[node distance=1.8cm, auto]
  \node (A) {\small$\BdDiff(\ttW_{g,1},\partial)_{\sigma}$};
  \node (B) [right of=A, node distance=3.2cm] {\small$\BdDiff(\ttW_{g+1,1},\partial)_{\sigma}$};
  \node (C) [below of=A] {\small$\BdDiff(\ttW_{g,1},\partial)$};  
  \node (D) [below of=B] {\small$\BdDiff(\ttW_{g+1,1},\partial)$};
  \draw[right hook->] (C) to node  {}(D);
  \draw [right hook->] (A) to node {}(C);
  \draw [right hook->] (A) to node {$$} (B);
  \draw [right hook->] (B) to node {$$} (D);
\end{tikzpicture}
\]
induces a map of pairs from $(${\small$\BdDiff(\ttW_{g+1,1},\partial)_{\sigma}$},{\small$\BdDiff(\ttW_{g,1},\partial)_{\sigma}$}$)$ to $(${\small$\BdDiff(\ttW_{g+1,1},\partial)$}, {\small$\BdDiff(\ttW_{g,1},\partial)$}$)$, which is homologically trivial in  degrees less than $(g-1)/2$.
\end{lem}
\begin{proof}
 We fix a class $[x]\in H_k(\dDiff(\tW_{g+1,1},\partial)_{\sigma},\dDiff(\tW_{g,1},\partial)_{\sigma})$ for $k\leq (g-2)/2$.
 We show that we can find a cycle representative for $[x]$ so that its image in {\small$H_k(\dDiff(\tW_{g+1,1},\partial),\dDiff(\tW_{g,1},\partial))$} factors through a class {\small$[y]\in H_{k-1}(\dDiff(\tW_{g,1},\partial)_{\sigma})$}. We then use induction to find a cycle representative for $[y]$ in order to show the image of $[y]$  in {\small$H_k(\dDiff(\tW_{g+1,1},\partial),\dDiff(\tW_{g,1},\partial))$} is zero.
 
$\textbf{Step 1:}$ First we find a ``good" cycle representative for $[x]$ to show that its image in $E^{\infty}$-page factors through  {\small$ H_{k-1}(\dDiff(\tW_{g,1},\partial)_{\sigma})$}.  For brevity, we shall write  $G_0, G_1, G_2$ and $G$ to denote the groups {\small$\dDiff(\tW_{g,1}, \partial)_{\sigma}, \dDiff(\tW_{g+1,1},\partial)_{\sigma},\allowbreak  \dDiff(\tW_{g,1},\partial)$} and {\small$\dDiff(\tW_{g+1,1},\partial)$} respectively.  A class $[x]$ in $H_k(G_1,G_0)$ is represented in the homogenous chain as a finite sum $\sum_ia_i(x_{0,i},x_{1,i}, \dots, x_{k,i})$ where $x_{j,i}\in G_1, a_i\in\bZ$ and the boundary $dx$ is a chain in $G_0$. We choose $\epsilon$ so that all $x_{j,i}$'s fix the $\epsilon$-collar.  Note that by \Cref{collar}, we can choose $x_{j,i}$'s to be the identity in an arbitrary collar neighborhood of  the boundary. The idea is to find $t\in G$ so that $t^{-1}x_{j,i}t\in G_2$ for all $i,j$. The diffeomorphism $x_{j,i}$ fix a neighborhood of $\phi_{\sigma}(C)$. We want to find a conjugation $t^{-1}x_{j,i}t$ that fixes $H\subset \W\natural\tW_{1,1}$.
 
 Let $e'$ be an orientation preserving embedding of  $[0,1]\times D^{2n-1}$ into $\tW_{g+1,1}\cong\W\natural\tW_{1,1}$ $$e':[0,1]\times D^{2n-1}\hookrightarrow \W \displaystyle\cup H= \W\natural\tW_{1,1}$$ such that  $ e'(0,D^{2n-1})=c(0,D^{2n-1})$ and $e'(1,D^{2n-1})=e(0,D^{2n-1})$.  We choose $e'$ in such a way that its image is in the $\epsilon$-collar neighborhood of the boundary. Let $\phi_0$ be a germ of an embedded core whose image lies in $$e'([0,1]\times D^{2n-1})\cup H$$ and it is in the same orbit as $\phi_{\sigma}$ i.e. $I(\phi_{\sigma})=I(\phi_0)$. Let $U$ be an open neighborhood of $\partial\W\cup H\cup \phi_{\sigma}(C)$ whose closure is diffeomorphic to $\tW_{2,2}$. Similar to \Cref{orbit}, we can use \cite[Corollary 4.4]{Galatius-Randal-Williams}, to find a diffeomorphism $t\in G$ that sends $[\phi_0]$ to $[\phi_{\sigma}]$ and whose support is in $U$. Hence, Let  $N$ and $N_0$ be subsets of $U$ and open neighborhoods of $\phi(C)$ and $\phi_0(C)$ respectively such that their closures are diffeomorphic to $H$. Let $t\in G$ be a diffeomorphism whose support is in $U$ and  $t(N_0)=N$. Since $H\backslash H\cap N_0$ is in a collar neighborhood of the boundary of $\W\natural \tW{1,1}$, by \Cref{collar} we may assume that $x_{j,i}$ are chosen for all $j,i$ so that $x_{j,i}$ fixes $t(H\backslash H\cap N_0)$ i.e.
 \[
t( H\backslash H\cap N_0)\subset \tW_{g+1,1}\backslash\text{support}(x_{j,i}).
 \]
 \begin{figure}[ht]
 \
\begin{tikzpicture}[scale=.3]

\begin{scope}[shift={(6,0)}]
\draw[line width=1.05pt] [dashed] (0,.-2.5) arc (-90:90:0.5 and 2.5);
\draw [line width=1.05pt] (0,2.5) arc (90:270:0.5 and 2.5);
\draw [line width=1.05pt] (0,0)+(0,-2.5) arc (-90:90:10.5 and 2.5);
\draw  [red, ultra thick](6,-0.37) arc (90:-90: 0.3 and .6);
\draw  [red, ultra thick, dashed](6,-0.37) arc (90:270: 0.3 and .6);
\draw  [red, ultra thick] (6,0) ellipse (2.1 and .9);
\draw  [red, ultra thick] (-0.5,0) to [out= 20, in=140] (2.5,0);
\draw  [red, ultra thick] (2.5,0) to [out= -40, in=190] (3.9,0);
\draw  [line width=1.05pt] (-.2,2.17)--(-2,2.17);
\draw  [line width=1.05pt] (-.35,1.5)--(-2.3,1.5);
\draw[line width=1.05pt] (-2,2.17) arc (90:270:0.3 and 1);
\draw[line width=1.05pt] (-2,2.17) arc (90:-90:0.3 and 1);
\draw[line width=1.05pt] (-2,2.17) arc (90:270:3 and 1);
\draw [white, thick] (-1.1, 2)--(-2,2);
\draw [white, thick] (-1.1, 1.7)--(-2,1.7);
\draw [white, line width=3.8] (0,1.84)--(-1,1.84);
\begin{scope}[shift={(-5.2,1.1)}, scale=.3]
\draw [line width=1.1pt] (4.6,.5) arc (240:300:2.5 and 6.75);
\draw   [line width=1.05pt] (6.87, 0.13) arc (48:150:1.2 and 0.9);
\end{scope}
\node at (-2.4,-3.3) {\tiny$H$};
\node at (6, -3.4) {\tiny$\W\backslash\text{\textcolor{red}{ an embedded core}}$};
\end{scope}
 \begin{scope}[shift={(-14,0)}]
\draw[line width=1.05pt] [dashed] (0,.-2.5) arc (-90:90:0.5 and 2.5);
\draw [line width=1.05pt] (0,2.5) arc (90:270:0.5 and 2.5);
\draw [line width=1.05pt] (0,0)+(0,-2.5) arc (-90:90:10.5 and 2.5);
\draw  [red, ultra thick](6,-0.37) arc (90:-90: 0.3 and .6);
\draw  [red, ultra thick, dashed](6,-0.37) arc (90:270: 0.3 and .6);
\draw  [red, ultra thick] (6,0) ellipse (2.1 and .9);
\draw  [red, ultra thick] (-0.5,0) to [out= 20, in=140] (2.5,0);
\draw  [red, ultra thick] (2.5,0) to [out= -40, in=190] (3.9,0);
\node at (6, -3.4) {\tiny$\W\backslash \text{\textcolor{red}{ an embedded core}}$};
\end{scope}
 \begin{scope}[shift={(-14,-10)}]
\draw[line width=1.05pt] [dashed] (0,.-2.5) arc (-90:90:0.5 and 2.5);
\draw [line width=1.05pt] (0,2.5) arc (90:270:0.5 and 2.5);
\draw [line width=1.05pt] (0,0)+(0,-2.5) arc (-90:90:10.5 and 2.5);
\node at (6, -3.4) {\tiny$\W$};
\end{scope}

\begin{scope}[shift={(6,-10)}]
\draw[line width=1.05pt] [dashed] (0,.-2.5) arc (-90:90:0.5 and 2.5);
\draw [line width=1.05pt] (0,2.5) arc (90:270:0.5 and 2.5);
\draw [line width=1.05pt] (0,0)+(0,-2.5) arc (-90:90:10.5 and 2.5);
\draw  [line width=1.05pt] (-.2,2.17)--(-2,2.17);
\draw  [line width=1.05pt] (-.35,1.5)--(-2.3,1.5);
\draw[line width=1.05pt] (-2,2.17) arc (90:270:0.3 and 1);
\draw[line width=1.05pt] (-2,2.17) arc (90:-90:0.3 and 1);
\draw[line width=1.05pt] (-2,2.17) arc (90:270:3 and 1);
\draw [white, thick] (-1.1, 2)--(-2,2);
\draw [white, thick] (-1.1, 1.7)--(-2,1.7);
\draw [white, line width=3.8] (0,1.84)--(-1,1.84);
\begin{scope}[shift={(-5.2,1.1)}, scale=.3]
\draw [line width=1.1pt] (4.6,.5) arc (240:300:2.5 and 6.75);
\draw   [line width=1.05pt] (6.87, 0.13) arc (48:150:1.2 and 0.9);
\end{scope}
\node at (6, -3.4) {\tiny$\tW_{g+1,1}$};

\end{scope}
\draw [right hook->,line width=1.05pt] (-3,0)--(1.5,0);
\draw [right hook->,line width=1.05pt] (-8.3,-4.2)--(-8.3,-6.9);
\draw [right hook->,line width=1.05pt] (8,-4.2)--(8,-6.9);
\draw [right hook->,line width=1.05pt] (-3,-10)--(1.5,-10);

\end{tikzpicture}  
 \caption{ Cartoon of a diagram that induces the maps in the diagram of  \Cref{relative} for  $n=1$}
  
 \end{figure}
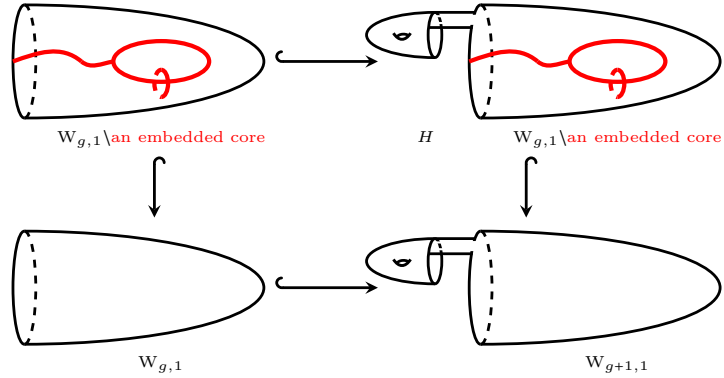

 With these choices, it is easy to see that for all $i$ and $j$, we have $t^{-1}x_{j,i}t\in G_2$ i.e. $t^{-1}x_{j,i}t$ fixes $H\subset \tW_{g+1,1}$. Note that $x_{j,i}$ fix the germ of $\phi(C)$, hence one can choose $N$ so that all $x_{j,i}$ fix $N$. By the choice of $t$, the element $t^{-1}x_{j,i}t$ is identity on $N_0$, and because $x_{j,i}$ fixes $t(H\backslash H\cap N_0)$, the element $t^{-1}x_{j,i}t$ is also identity on $H\backslash H\cap N_0$. Hence $t^{-1}x_{j,i}t\in G_2$ for all $i$ and $j$. 
 
 Following the notation of \cite[Lemma 2.5.]{wahl2010homological}, we write $(x_{0,i}, x_{1,i},\dots, x_{k,i})\times t$ to denote  the $(k+1)$-chain given by the following linear combination
  \[
  (x_{0,i}, x_{1,i},\dots, x_{k,i}, t) + \sum_{j=0}^k(x_{0,i}, x_{1,i},\dots, x_{j-1,i},t, t^{-1}x_{j,i}t, t^{-1}x_{j+1,i}t, \dots, t^{-1}x_{k,i}t)
  \]
  
 It is easy to see $d(x\times t)= (-1)^k x+dx\times t+(-1)^{k+1}t^{-1}xt$. Let $t\in G$ be the diffeomorphism depending on $x$ such that conjugation by $t$ maps all $x_{j,i}$'s to $G_2$. For such $t$, we have $ [t^{-1}xt]=0$ in $H_k(G, G_2)$ that implies the image of the class $[x]$ in  $H_k(G, G_2)$ is equal to $(-1)^{k-1}[dx\times t]$. We obtain a non-canonical factorization of the relative map $H_k(G_1,G_0)\rightarrow H_k(G, G_2)$ via $H_{k-1}(G_0)$.

$\textbf{Step 2:}$ Recall that we want to prove that
 \[
H_k(G_1,G_0)\rightarrow  H_k(G, G_2)
\]
is zero as long as $k\leq (g-2)/2$. Let $[\phi_{\sigma'}]$ be a germ of an embedded core in $\W$ disjoint from $[\phi_{\sigma}]$. By induction, we know for $k\leq (g-4)/2$
\[
H_{k-1}(\dDiff(\tW_{g,1},\partial)_{(\sigma, \sigma')})\rightarrow H_{k-1}(\dDiff(\tW_{g,1},\partial)_{\sigma})=H_{k-1}(G_0)
\]
is an isomorphism where $\dDiff(\tW_{g,1},\partial)_{(\sigma, \sigma')}$ is the stabilizer subgroup fixing $\phi_{\sigma}$ and $\phi_{\sigma'}$. 

Hence,  $dx$ is homologous to a class  {\small$y\in H_{k-1}(\dDiff(\tW_{g,1},\partial)_{(\sigma, \sigma')})$} meaning that it can be represented by a linear combination $k$-tuples of elements in $\dDiff(\tW_{g,1},\partial)$ that fix a neighborhood of $\phi_{\sigma}(C)$ and $\phi_{\sigma'}(C)$. By the \Cref{collar}, we can assume all diffeomorphisms appearing in the cycle representative of $y$ fix the support of $t$. Let $U'$ be a neighborhood of $\phi_{\sigma'}(C)$ that is fixed by all diffeomorphisms appearing in cycle representative of $y$. Since $dy=0$ and  $t^{-1}yt=y$, for any diffeomorphism $t'\in G$, we have
\begin{align*}
d(y&\times t\times t')\\&=(-1)^{k+1}(y\times t) +d(y\times t)\times t' + (-1)^kt'^{-1}(y\times t)t'\\&=(-1)^{k+1}(y\times t) +(-1)^{k}(y\times t')+(dy\times t)\times t'+ (-1)^{k+1}(t^{-1}yt)\times t' +(-1)^kt'^{-1}(y\times t)t'\\&=  (-1)^{k+1}(y\times t) +(-1)^{k}(y\times t')+ (-1)^{k+1}(t^{-1}yt)\times t' +(-1)^kt'^{-1}(y\times t)t'\\&=(-1)^{k+1}(y\times t) +(-1)^kt'^{-1}(y\times t)t'.
\end{align*}
Hence $[x]=[y\times t]=[t'^{-1}(y\times t)t']$ in $H_k(G,G_2)$. To finish the proof, we need to find $t'$ such that  $[t'^{-1}(y\times t)t']=0$ in $H_k(G,G_2)$. Suppose $y=\sum b_i(y_{0,i},\dots,y_{k-1,i}) $ then by definition $t'^{-1}(y\times t)t'$ is
\begin{align*}
\sum_i(t'^{-1}&y_{0,i}t',\dots, t'^{-1}y_{k-1,i}t', t'^{-1}tt') + \\ &\sum_i \sum_{j=0}^{k-1}(t'^{-1}y_{0,i}t',\dots,t'^{-1}tt',t'^{-1} t^{-1}y_{j,i}tt', t'^{-1}t^{-1}y_{j+1,i}tt', \dots, t'^{-1}t^{-1}y_{k,i}tt').
\end{align*}

Similar to step $1$, let $\phi_1$ be a germ of an embedded core whose image lies in $$e'([0,1]\times D^{2n-1})\cup H$$ and it is in the same orbit as $\phi_{\sigma'}$ i.e. $I(\phi_{\sigma'})=I(\phi_1)$. Let $U'$ be an open neighborhood of $\partial\W\cup H\cup \phi_{\sigma'}(C)$ whose closure is diffeomorphic to $\tW_{2,2}$. Similar to step 1, we can use \cite[Corollary 4.4]{Galatius-Randal-Williams}, to find a diffeomorphism $t'\in G$ that maps $[\phi_1]$ to $[\phi_{\sigma'}]$ and whose support is in $U'$. Hence, Let  $N'$ and $N_1$ be subsets of $U'$ and open neighborhoods of $\phi_{\sigma'}(C)$ and $\phi_1(C)$ respectively such that their closures are diffeomorphic to $H$.

 Let $t'\in G$ be a diffeomorphism whose support is in $U'$ and  $t'(N_1)=N'$. Since $y_{j,i}\in \dDiff(\tW_{g,1},\partial)_{(\sigma, \sigma')}$, we can arrange $t'$ so that  $t'^{-1}y_{j,i}t'=y_{j,i}$ for all $i,j$. Therefore to show that $[t'^{-1}(y\times t)t']=0$ in $H_k(G,G_2)$, it is left to prove that $t'$ can be chosen so that $t'^{-1}tt'\in G_2$. But this also can be done similar to step 1. Note that $H\backslash H\cap N_1$ is contained in a collar neighborhood of $\partial\tW_{g+1,1}$. Also note that since $t$ is in  $\dDiff(\tW_{g+1,1},\partial)$, it fixes a collar neighborhood of the boundary. Hence we can choose $N_1$ so that $t'(H\backslash H\cap N_1)$ is fixed by $t$, therefore $t'^{-1}tt'$ fixes $H\subset \W\natural \tW_{1,1}$ i.e. $t'^{-1}tt'\in G_2$. 
\end{proof}
  As always after proving homological stability for a family of groups, the next step is to study the limit. Consider the following space 
  \[
  \mathcal{M} :=\coprod_{g} \BdDiff(\W,\partial).
  \] 
  It is not hard to see that this space is an H-space with an associative and commutative product up to homotopy, but it is not clear to the author whether $\mathcal{M}$ has an $A_{\infty}$-structure. We would like to understand 
  \[
  \textit{hocolim }(\mathcal{M}\xrightarrow{.\tW_{1,1}}\mathcal{M}\xrightarrow{.\tW_{1,1}}\dots)
  \]
 where the product by $.\tW_{1,1}$ is the standard model for the stabilization map. To this end, in the next section we prove that there exists a certain infinite loop space whose homology groups compute the stable homology of $\BdDiff(\W,\partial)$.


\section{Stable moduli of flat bundles}\label{stablehomology}

 In this section, we will prove \Cref{limit homology} and \Cref{splitting}. In order to study the map  
 \[
\iota: \BdDiff(\tW,\partial)\to \BDiff(\tW,\partial)
 \]
 which is induced by the identity homomorphism, Thurston in \cite{thurston1974foliations} studied the homotopy fiber of this map denoted by $\overline{\BDiff(\tW,\partial)}$.  The space $\overline{\BDiff(\tW,\partial)}$ classifies trivial $\tW$-bundles with a foliation transverse to the fibers. Thurston proved a homological h-principle theorem in \cite{thurston1974foliations} that the space $\overline{\BDiff(\tW,\partial)}$ is homology equivalent to a section space of certain fiber bundle over $\tW$. On the other hand, Galatius and Randal-Williams developed the parametrized surgery theory in \cite{galatius2012stable} to study $\BDiff(\tW,\partial)$ for certain manifolds $\tW$. We will combine these two methods to study the stable homology of $\BdDiff(\WW,\partial)$. Let us first briefly digress  to explain Thurston's theorem.
 \subsection{Recollection from foliation theory}  We follow McDuff's exposition (\cite{mcduff1979foliations}) on her joint work with Segal in foliation theory. Let $\tW$ be an $n$-dimensional smooth manifold possibly with boundary. The smooth structure on $\tW$ can be considered as smooth codimension $n$ foliation by points whose normal bundle is the tangent bundle of $\tW$. Hence by Haefliger's theorem \cite{haefliger1971homotopy}, the smooth structure on $\tW$ gives rise to a homotopy commutative diagram
  \[
 \begin{tikzpicture}[node distance=1.8cm, auto]
  \node (A) {$\text{W}$};
  \node (B) [right of=A] {$\mathrm{BGL}^+_{n}(\bR)$};
  \node (C) [above of= B ] {$\mathrm{BS}\Gamma_{n}$};  
   \draw [->] (A) to node {$\tau$}(B);
  \draw [->] (C) to node {$\nu$}(B);
  \draw [->] (A) to node {$f$}(C);
\end{tikzpicture}
 \]
 where $\tau$ classifies the tangent bundle. We want to study the space of lifts of the tangent bundle to $\mathrm{BS}\Gamma_{n}$.  The map $\nu$ might not be a Serre fibration. We turn $\nu$ into a Serre fibration and we denote by $\tau^*(\nu)$ the pullback of this fibration via $\tau$. The space of sections of $\tau^*(\nu)$ is endowed with the compact-open topology and has a base section given by the point foliation. More explicitly, this model for the space of sections is to take  the space of all pairs $(g,h)$ where $g$ is a map $g:\text{W}\rightarrow \mathrm{BS}\Gamma_{n}$ and $h$ is a homotopy from $\tau$ to $\nu\circ g$. The base point for this model is $s_0=(f, h_0)$ where $h_0$ is the trivial homotopy. If $\tW$ has a nonempty boundary, let $\mathcal{Sect}(\text{W},\partial)$ be the space of sections over $\text{W}$ that is equal to the base section $s_0$ in a germ of a collar  of the boundary.  
 
 \begin{thm}[Thurston \cite{mcduff1979foliations}]\label{th}
 There exists a map $$f_{\text{\textnormal{W}}}:  \overline{\BDiff(\text{\textnormal{W}},\partial)}\rightarrow \mathcal{S}(\text{\textnormal{W}},\partial)$$ that induces a homology isomorphism.
\end{thm}
 
 Our goal in this section is to describe a model for the map $f_{\text{\textnormal{W}}}$ that is $\Diff(\tW,\partial)$-equivariant. Let us first describe a model for $\mathcal{S}(\text{\textnormal{W}},\partial)$ that admits a $\Diff(\tW,\partial)$-action.  Let $\gamma$ be the tautological bundle  on $\mathrm{BGL}^+_{n}(\bR)$.  Let us assume that we fixed a model for the map  $\nu:\mathrm{BS}\Gamma_{n}\to \mathrm{BGL}_n^+(\bR)$ that is a Serre fibration. Let $\nu^*(\gamma)$ be the pullback bundle  over $\mathrm{BS}\Gamma_{n}$. We denote  by $\text{Bun}_{\partial}(T\tW,{\nu}^*\gamma)$ the space of all bundle maps $T\tW\rightarrow {\nu}^*\gamma$ from the tangent bundle of $\tW$ to ${\nu}^*(\gamma)$ that are standard on a germ of a collar of the boundary and equipped with the compact-open topology (See  \cite[Section 1.1]{randal2009resolutions} for more details). The action of $\Diff(\tW,\partial)$  is given by precomposing a bundle map with the differential of a diffeomorphism.
 
 For a model that $\nu$ is a Serre fibration, one can define a map between $ \mathcal{Sect}(\tW,\partial)$ and $\text{Bun}_{\partial}(T\tW,{\nu}^*\gamma)$ as follows. First fix an isomorphism between $T\tW$ and $\tau^*\gamma$. Every section $s\in  \mathcal{Sect}(\tW,\partial)$ gives a map $s:\tW\to \mathrm{BS}\Gamma_{n}$ such that ${\nu}\circ s=\tau$. Hence, we obtain a bundle map $s^*: T\tW\to {\nu}^*\gamma$. It is easy prove that  the map that associates a bundle map to a section
 \[
\alpha: \mathcal{Sect}(\tW,\partial)\to \text{Bun}_{\partial}(T\tW,{\nu}^*\gamma)
 \]
 is a weak homotopy equivalence.

  The map $f_{\text{\textnormal{W}}}$ in \Cref{th} is roughly described as follows.  Let $\tW\hcoker \Diff(\tW,\partial)$ be the homotopy quotient of the action of $\Diff(\tW,\partial)$ on $\tW$ which is the canonical manifold bundle over $\BDiff(\tW,\partial)$ with fiber $\tW$. If we pull this bundle back to $\BdDiff(\tW,\partial)$, its structural group is discrete, so we may consider it as a foliated bundle (i.e. there is a Haefliger structure of codimension $n$ transverse to the fibers). If we further pull it back to the space $\overline{\BDiff(\tW,\partial)}$, we obtain a trivial fibration $\overline{\BDiff(\tW,\partial)}\times \tW$. Hence, on this bundle there is a canonical Haefliger structure of codimension $n$ transverse to the fibers $\{ x\}\times \tW$. The normal bundle of this Haefliger structure is induced by the tangent bundle $T\tW$. Therefore, this structure is classified by a homotopy commutative diagram 
   \begin{equation}\label{thurston}
   \begin{gathered}
 \begin{tikzpicture}[node distance=2.9cm, auto]
  \node (A) {$\text{W}$};
  \node (B) [right of=A] {$\mathrm{BGL}_{n}^+(\bR)$};
  \node (C) [above of= B, node distance=1.8cm ] {$\mathrm{BS}\Gamma_{n}$};  
  \node (D) [above of= A, node distance=1.8cm] {$ \overline{\BDiff(\tW,\partial)}\times \text{W}$};
   \draw [->] (A) to node {$\tau$}(B);
  \draw [->] (C) to node {$\nu$}(B);
  \draw [->] (D) to node {$pr_2$}(A);
  \draw [->] (D) to node {$F$} (C);
\end{tikzpicture}
\end{gathered}
\end{equation}
where $pr_2$ is the projection to the second factor. Given a choice of homotopy $H$ that makes the above diagram commutative, for each point $x \in \overline{\BDiff(\tW,\partial)}$, the map $F|_{x\times \tW}$ gives a homotopy lift of $\tau$ to $\mathrm{BS}\Gamma_n$ with a choice of homotopy $H|_{x\times \tW}$. Thus, this construction defines a map
\[
f_{\tW}: \overline{\BDiff(\tW,\partial)}\to \mathcal{Sect}(\tW,\partial).
\]

 We now describe certain categorical models for $\overline{\BDiff(\tW,\partial)}$, $\mathrm{BS}\Gamma_{n}$ and $\mathrm{BGL}_{n}^+(\bR)$ in order to have a $\Diff(\tW,\partial)$-equivariant model for the map $f_{\tW}$.
  \begin{defn}
Let $\mathcal{M}$ be a topological monoid that acts on a space $X$ from the left.  We shall write $\mathcal{C}(\mathcal{M}\hker X)$ to denote the topological category whose space of objects is $X$ and whose space of morphisms is $\mathcal{M}\times X$, where $(m,x)$ corresponds to the morphism $x\rightarrow mx$. 
\end{defn}
Our model for $\overline{\BDiff(\tW,\partial)}$ is the fat realization of $\mathcal{C}(\dDiff(\tW,\partial)\hker \Diff(\tW,\partial))$ which also admits an action of $\Diff(\tW,\partial)$ from the right. 
\begin{defn} Let $\Gamma(\tW,\partial)$ be the topological category whose space of objects is $\text{int}(\tW)=\tW\backslash \partial\tW$ with its usual topology and whose space of morphisms from $x$ to $y$ is the set of germs of local orientation preserving diffeomorphims of $\tW\backslash\partial\tW$ that send $x$ to $y$. The morphism space of $ \Gamma(\tW,\partial)$ is equipped with the sheaf topology i.e. an open neighborhood of a germ of a diffeomorphism that sends $x$ to $y$ is described as follows: let $U$ and $V$ be open sets containing $x$ and $y$ respectively and let $g:U\to V$ be a local diffeomorphism that sends $x$ to $y$, the set of germs of $g$ at all points in $U$ is an open neighborhood of the germ of $g$ at $x$. 
\end{defn}
It is shown in \cite[Lemma 1]{mcduff1979foliations} that the fat realization of the category $\Gamma(\tW,\partial)$ is homotopy equivalent to $\mathrm{BS}\Gamma_n$.
\begin{defn}
Let $\mathrm{GL}(\tW,\partial)$ be the topological category which also has the interior of the manifold i.e. $\text{int}(\tW)=\tW\backslash \partial\tW$ as the space  of objects, and the morphism space is a bundle over $\tW\times \tW$ whose fiber over $(x,y)$ is the space of all orientation preserving linear isomorphisms with its usual topology from tangent space $T_x$ at the point $x$, to the tangent space  $T_y$ at the point $y$.
\end{defn}
McDuff also showed in \cite[Lemma 2]{mcduff1979foliations} that the realization of $\mathrm{GL}(\tW,\partial)$ is homotopy equivalent to $\mathrm{BGL}^+_n(\bR)$. Note that there exists a functor $\tilde{\nu}:\Gamma(\tW,\partial)\rightarrow \mathrm{GL}(\tW,\partial)$ that is the identity on the space of objects and it sends the morphism $f:x\rightarrow y$ to its derivative $\text{d}f|_x:x\rightarrow y$.

Following \cite{mcduff1979foliations}, one can describe the categorical model of the diagram \ref{thurston} as follows. The  group $\dDiff(\tW,,\partial)$ acts on $\Diff(\tW,\partial)\times \tW$ from the left by letting $g\in \dDiff(\tW,\partial)$ act as $g: (f,x)\rightarrow (gf,x).$ Hence, there exists a functor 
\[\tilde{F}: \mathcal{C}(\dDiff(\tW,\partial)\hker \Diff(\tW,\partial)\times \tW)\rightarrow \Gamma(\tW,\partial)\]
that takes the object $(f,x)$ to $f(x)$ and the morphism $g:(f,x)\rightarrow (gf,x)$ to the germ of $g$ at $f(x)$. The following diagram is a model for the categorification of the diagram \ref{thurston}
\begin{equation}\label{categorification}
  \begin{gathered}
 \begin{tikzpicture}[node distance=3.7cm, auto]
  \node (A) {$\mathcal{C}(e\hker \tW)$};
  \node (B) [right of=A] {$\mathrm{GL}(\tW,\partial)$};
  \node (C) [above of= B, node distance=1.8cm ] {$\Gamma(\tW,\partial)$};  
  \node (D) [above of= A, node distance=1.8cm] {$ \mathcal{C}(\dDiff(\tW,\partial)\hker \Diff(\tW,\partial)\times \tW)$};
   \draw [->] (A) to node {$\tilde{\tau}$}(B);
  \draw [->] (C) to node {$\tilde{\nu}$}(B);
  \draw [->] (D) to node {$\tilde{pr_2}$}(A);
  \draw [->] (D) to node {$\tilde{F}$} (C);
\end{tikzpicture}
  \end{gathered}
\end{equation}
where $\mathcal{C}(e\hker \tW)$ is the category that arises from the action of the trivial group $\{ e\}$ on $\tW$, the functor $\tilde{pr_2}$ is induced by the obvious projection, and $\tilde{\tau}$ are induced by the identity map on the space of objects. Note that $\tilde{\tau}\circ \tilde{pr_2}\neq \tilde{\nu}\circ \tilde{F}$. However, there exists a natural transformation $\tilde{H}: \tilde{\nu}\circ \tilde{F} \rightarrow \tilde{\tau}\circ \tilde{pr_2}$ that sends the object $(f,x)$ of $\mathcal{C}(\dDiff(\tW,\partial)\hker \Diff(\tW,\partial)\times \tW)$ to the morphism of the category $\mathrm{GL}(\tW,\partial)$ as follows
\[
\text{d}f^{-1}|_{f(x)}: f(x)=\tilde{\nu}\circ \tilde{F}(f,x) \rightarrow x=\tilde{\tau}\circ \tilde{pr_2}(f,x).
\]

Since natural transformations induce homotopies after  realization,  the realization of the diagram \ref{categorification} similar to the diagram \ref{thurston} provides us with  a map $f_{\tW}$ from  $|| \mathcal{C}(\dDiff(\ttW,\partial)\hker \Diff(\ttW,\partial))||$, which is our model for $\overline{\BDiff(\tW,\partial)}$, to the section space $\mathcal{S}(\tW,\partial)$. The advantage of this functorial diagram with respect to $\tW$, as we shall see, is it gives a $\Diff(\tW,\partial)$-equivariant model for $f_{\tW}$.
\begin{prop}
The map $$f_{\ttW}: \overline{\BDiff(\ttW,\partial)} \to \mathcal{S}(\ttW,\partial)$$ is $\Diff(\ttW,\partial)$-equivariant. 
\end{prop}
\begin{proof}
The group $\Diff(\tW,\partial)$ acts on $\overline{\BDiff(\tW,\partial)}$ by acting on the space of objects of the category $|| \mathcal{C}(\dDiff(\ttW,\partial)\hker \Diff(\ttW,\partial))||$ from the right. Let $b\in \overline{\BDiff(\tW,\partial)}$ and $g\in \Diff(\tW,\partial)$. To prove that $f_{\tW}$ is equivariant, we show that $f_{\tW}(g\cdot b)$ is the bundle map induced by $f_{\tW}(b)$ precomposed with the action of $g$ on the tangent bundle $T\tW$.

 Recall that $f_{\tW}$ is induced by the diagram \ref{categorification} and a canonical homotopy induced by the natural transformation $\tilde{H}$. To determine $f_{\tW}(g\cdot b)$ as a bundle map, we first describe the action of $\Diff(\tW,\partial)$ on the topological categories in  the diagram \ref{categorification}. An element $g\in \Diff(\tW,\partial)$ acts on the space of objects of the category $ \mathcal{C}(\dDiff(\tW,\partial)\hker \Diff(\tW,\partial)\times \tW)$ by sending $(f,x)$ to $(f\cdot g^{-1}, g(x))$ and it acts trivially on its space of morphisms. Let $\Diff(\tW,\partial)$ act on the categories $\Gamma(\tW,\partial)$ and $\mathrm{GL}(\tW,\partial)$ trivially and let it act on the category $\mathcal{C}(e\hker \tW)$ in the obvious way.

 We denote the action of $g$ on $\mathcal{C}(e\hker \tW)$ as a functor $\tilde{i}_g: \mathcal{C}(e\hker \tW)\to \mathcal{C}(e\hker \tW)$. Consider the following diagram of categories
\begin{equation}\label{categorification2}
  \begin{gathered}
 \begin{tikzpicture}[node distance=3.7cm, auto]
  \node (A) {$\mathcal{C}(e\hker \tW)$};
  \node (B) [right of=A] {$\mathrm{GL}(\tW,\partial).$};
  \node (C) [above of= B, node distance=1.8cm ] {$\mathrm{GL}(\tW,\partial)$};  
  \node (D) [above of= A, node distance=1.8cm] {$\mathcal{C}(e\hker \tW)$};
   \draw [->] (A) to node {$\tilde{\tau}$}(B);
  \draw [->] (C) to node {$\text{Id}$}(B);
  \draw [->] (D) to node {$\tilde{i}_g$}(A);
  \draw [->] (D) to node {$\tilde{\tau}$} (C);
\end{tikzpicture}
  \end{gathered}
\end{equation}
This diagram is not commutative but there is a natural transformation $\tilde{H'}$ from $\text{Id}\circ \tilde{\tau}$ to $\tilde{\tau}\circ\tilde{i}_g$. We describe $\tilde{H'}$ as a continuous map from space of objects in $\mathcal{C}(e\hker \tW)$ to space of morphisms in $\mathrm{GL}(\tW,\partial)$ as follows
\[
x\rightarrow \text{d}g|_x.
\]
By geometrically realizing the diagram \ref{categorification2}, we obtain the action of $g$ on the tangent bundle $T\tW$ given by the differential of $g$.

By definition of the map $f_{\tW}$, the map $f_{\tW}(g\cdot\text{---})$ is induced by the diagram
\begin{equation}\label{categorification3}
  \begin{gathered}
 \begin{tikzpicture}[node distance=3.7cm, auto]
  \node (A) {$\mathcal{C}(e\hker \tW)$};
  \node (B) [right of=A] {$\mathrm{GL}(\tW,\partial)$};
  \node (C) [above of= B, node distance=1.8cm ] {$\Gamma(\tW,\partial)$};  
  \node (D) [above of= A, node distance=1.8cm] {$ \mathcal{C}(\dDiff(\tW,\partial)\hker \Diff(\tW,\partial)\times \tW)$};
   \draw [->] (A) to node {$\tilde{\tau}$}(B);
  \draw [->] (C) to node {$\tilde{\nu}$}(B);
  \draw [->] (D) to node {$\tilde{i}_g\circ\tilde{pr_2}$}(A);
  \draw [->] (D) to node {$\tilde{F}$} (C);
\end{tikzpicture}
  \end{gathered}
\end{equation}
and a natural transformation $\tilde{H}_g :\tilde{\nu}\circ \tilde{F} \rightarrow \tilde{\tau}\circ \tilde{i}_g\circ\tilde{pr_2}$.  The natural transformation $\tilde{H}_g$ is induced by the continuous map from the space of objects of the category $\mathcal{C}(\dDiff(\tW,\partial)\hker \Diff(\tW,\partial)\times \tW)$ to the morphism space of $\mathrm{GL}(\tW,\partial)$ given as follows
\[
\text{d}g\circ \text{d}f^{-1}|_{f(x)}: f(x)=\tilde{\nu}\circ \tilde{F}(f, x)\to \tilde{\tau}\circ \tilde{i}_g\circ\tilde{pr_2}(f,x)=g(x).
\]
Hence, the bundle map $f_{\tW}(g\cdot\text{---})$ is the bundle map $f_{\tW}(\text{---})$ precomposed by the action of $g$ on the $T\tW$.
\end{proof}
\subsection{On stable moduli space of  flat $\W$-bundles}
Recall that we want to find a space whose homology is the same as the stable homology of $\BdDiff(\W,\partial)$.  Consider the following fibration sequence
 \[
  \begin{tikzpicture}[node distance=3.2cm, auto]
  \node (A) {$\overline{\text{B}\Diff(\W,\partial)}$};
  \node (B) [right of=A] {$\BdDiff(\W, \partial)$};
  \node (C) [right of= B ] {$\BDiff(\W, \partial).$};  
  \draw [->] (A) to node {$$}(B);
  \draw [->] (B) to node {$\iota$}(C);
\end{tikzpicture}
 \]
Note that an appropriate model for the homotopy fiber of $\iota$ admits an action of the topological group $\Diff(\WW,\partial)$ (e.g. the pullback of the universal $\Diff(\WW,\partial)$-bundle via $\iota$). Hence we have a map
\[
\overline{\text{B}\Diff(\W,\partial)}\hcoker \Diff(\WW,\partial)\xrightarrow{\simeq} \BdDiff(\W, \partial)
\]
which is a weak homotopy equivalence.   Consider the following homotopy commutative diagram 

  \[
 \begin{tikzpicture}[node distance=5.4cm, auto]
  \node (A) {$\overline{\text{B}\Diff(\W,\partial)}$};
  \node (B) [below of=A, node distance=2cm] {$\overline{\text{B}\Diff(\W,\partial)}\hcoker \Diff(\WW,\partial)$};
  \node (C) [below of= B, node distance=2cm ] {$\BDiff(\W, \partial)$};  
  \node (D) [right of= A] {$ \text{\textnormal{Bun}}_{\partial}(T\WW,\nu^*\gamma)$};
    \node (E) [right of= B] {$ \text{\textnormal{Bun}}_{\partial}(T\WW,\nu^*\gamma)\hcoker \Diff(\W, \partial)$};
  \node (F) [right of= C ] {$\BDiff(\W, \partial)$};  
   \draw [->] (A) to node {$f_{\W}$}(D);
  \draw [->] (B) to node {$$}(E);
  \draw [->] (C) to node {$=$}(F);
  \draw [->] (A) to node {$$}(B);
  \draw [->] (B) to node {$$} (C);
 \draw [->] (D) to node {$$} (E);
  \draw [->] (E) to node {$$} (F);

\end{tikzpicture}
\]

Because the equivariant map between fibers is homology isomorphism, by Thurston's theorem, we have the following corollary.

\begin{cor}\label{cor2}
There is a model for the classifying space $\BdDiff(\WW, \partial)$ and a map
\[
\BdDiff(\WW, \partial)\to \text{\textnormal{Bun}}_{\partial}(T\WW,\nu^*\gamma)\hcoker \Diff(\WW, \partial)
\]
that induces a homology isomorphism.
\end{cor}

 Hence, to describe the stable homology of $\BdDiff(\WW, \partial)$, as in \Cref{cor2} we replace it with the homotopy quotient of the bundle maps by diffeomorphisms which is a  better object to study from homotopy theory point of view. In fact,  Galatius and Randal-Williams in \cite{galatius2012stable} developed {\it parametrized surgery theory} to study such homotopy quotients. 
 
 Let us recall what the main theorem of Galatius and Randal-Williams in \cite{galatius2012stable} says for certain tangential structures that we are interested in.
\begin{defn} {\it The moduli space of tangential $\nu$-structure} on $\W$ is the space $\text{\textnormal{Bun}}_{\partial}(T\W,\nu^*\gamma)\hcoker \Diff(\W, \partial)$. For brevity, we shall denote this moduli space by $\mathrm{BDiff}^{\nu}(\W,\partial)$.
\end{defn} 
\begin{defn} For a tangential structure $\beta:B\rightarrow \mathrm{BO}(2n)$, the Madsen-Tillmann spectrum  associated to the map $\beta$ is the Thom spectrum of the virtual bundle $\beta^*(-\gamma)$ and is denoted by $\bold{MT}\beta$ (if we want to keep track of the base we denote this spectrum by $B^{-\beta}$). With abuse of notation, we denote the Madsen-Tillmann spectrum associated to $\mathrm{BSO}(2n)\rightarrow \mathrm{BO}(2n)$ by $\mathrm{BSO}(2n)^{-\gamma}$.
\end{defn}
\begin{defn}
Let $B$ be a path connected space. A tangential structure $\beta:B\rightarrow \mathrm{BO}(2n)$ is called {\it spherical} if the sphere $S^{2n}$ admits a $\beta$-structure.
\end{defn}
The main theorem in \cite[Theorem 1.8]{galatius2012stable} that we recall below is about spherical tangential structures and the tangential structure $\nu: \mathrm{BS}\Gamma_{2n}\rightarrow \mathrm{BGL}_{2n}^+(\bR)$ that we are interested in is indeed spherical because the foliation by points on $S^{2n}$ is a codimension $2n$ Haefliger structure whose normal bundle is the tangent bundle of $S^{2n}$, thus it gives a $\nu$-structure on $S^{2n}$.

 Recall from introduction that $\bold{MT}\nu^n$ is the Thom spectrum associated to the tangential structure $\nu^n :  \mathrm{B}\Gamma_{2n}\langle n\rangle\rightarrow \mathrm{BS}\Gamma_{2n}\rightarrow \mathrm{BGL}_{2n}^+(\bR)$ where $\mathrm{B}\Gamma_{2n}\langle n\rangle$ is the $n$-connected cover of $\mathrm{B}\Gamma_{2n}$.
To prove  \Cref{limit homology}, given \Cref{cor2}, we need  to show that 
 there exists a map 
 \[
\mathrm{BDiff}^{\nu}(\W,\partial)\rightarrow \Omega_0^{\infty}\bold{MT}\nu^n
 \]
 which is a homology isomorphism in the stable range.

%

 Let us  fix a $\nu$-structure on $S^{2n-1}$, then one can  define  the moduli space of highly connected bordism as \cite[ Definition 1.4]{galatius2012stable} denoted by $\mathcal{N}^\nu(S^{2n-1})$. With this notation, we have
\[
\mathcal{N}^\nu(S^{2n-1})\simeq\coprod_{\tW}\mathcal{S}(\tW, \partial)\hcoker \Diff(\tW, \partial),
\]
 where the disjoint union is over compact manifolds $\tW$ with $\partial\tW= S^{2n-1}$ such that $(
\tW, S^{2n-1})$ is $(n-1)$-connected, one in each diffeomorphism class.

 Let $\tW_{1,2}\subset [0,1]\times \mathbb{R}^\infty$ be an embedding as a cobordism with collar boundary.  Since $\nu$ is   $(2n+2)$-connected (\cite[Remark 1]{haefliger1971homotopy}), one can choose a $\nu$-structure on $\tW_{1,2}$ extending our chosen $\nu$-structures on $\{ 0\} \times S^{2n-1}$ and $\{ 1\} \times S^{2n-1}$. There exists an induced self-map $ \mathcal{N}^\nu(S^{2n-1})\rightarrow \mathcal{N}^\nu(S^{2n-1})$ defined by taking union with $\tW_{1,2}$ and subtracting $1$ from the first coordinate. Hence, we have the following commutative diagram
\[
 \begin{tikzpicture}[node distance=2.2cm, auto]
  \node (A) {$H_*(\coprod_{g} \BdDiff(\W,\partial);\bZ)$};
  \node (B) [below of=A]{$H_*(\coprod_{g} \BdDiff(\W,\partial);\bZ)$};
  \node (C) [right of=A, node distance=4.5cm]{$H_*(\mathcal{N}^\nu(S^{2n-1});\bZ)$};
  \node (D) [right of=B, node distance=4.5cm]{$H_*(\mathcal{N}^\nu(S^{2n-1});\bZ)$};
  \draw [->] (A) to node {$$} (C);
  \draw [->] (B) to node {$$} (D);
  \draw [->] (A) to node {$\coprod \tW_{1,2}$} (B);
  \draw [->] (C) to node {$\coprod \tW_{1,2}$} (D);
 \end{tikzpicture}
 \]
where the left vertical map is induced by the standard stabilization map and horizontal maps are induced by Thurston's theorem, which is a homology isomorphism. 

\begin{defn}\label{defn1} Assume that $K\subset [0,\infty)\times \mathbb{R}^\infty$ is a submanifold with a $\nu$-structure $l_K$, such that the first coordinate $ x_1: K\rightarrow [0,\infty)$ has the natural numbers as regular values and $ K|_{[i,i+1]}$ is a cobordism such that the pairs $( K|_{[i,i+1]},  K|_i)$ and $ (K|_{[i,i+1]}, K|_{i+1})$ are $(n-1)$-connected for all natural numbers $i$. The definition of the universal $\nu$-end $K$ in \cite[ Addendum 1.9]{galatius2012stable} is equivalent to satisfying  the following conditions:
\begin{itemize}
\item For each integer $i$, the map $\pi_n(K|_{[i,\infty)})\rightarrow \pi_n(\mathrm{BS}\Gamma_{2n})$ is surjective, for all base points in $K$.
\item For each integer $i$, the map $\pi_{n-1}(K|_{[i,\infty)})\rightarrow \pi_{n-1}(\mathrm{BS}\Gamma_{2n})$ is injective, for all base points in $K$.
\item For each integer $i$, each path component of $K|_{[i,\infty)}$ contains a submanifold diffeomorphic to $ S^n\times S^n - \text{int}(D^{2n})$, which in addition has null-homotopic structure map to $\Gamma_{2n}$.

\end{itemize}
\end{defn}
Using the main theorem of \cite[Theorem 1.8]{galatius2012stable} for the map $\nu: \mathrm{BS}\Gamma_{2n}\rightarrow \mathrm{BGL}_{2n}^+(\bR)$, we obtain
\begin{thm}\label{universalend}
 Let $2n>4$ and $(K, l_K)$ be a universal $\nu$-end such that $\mathcal{N}^{\nu}(K|_0,l_K|_0)\neq \emptyset$, then there is a homology equivalence 
 \[
 \text{\it{hocolim}}_{i\rightarrow \infty}\mathcal{N}^{\nu}(K|_i,l_K|_i)\rightarrow \Omega^{\infty} \bold{MT}\eta
 \]
 where $\eta: B'\rightarrow \mathrm{B}\Gamma_{2n}\rightarrow \mathrm{BGL}_{2n}^+(\bR)$ is the $n$th stage of the Moore-Postnikov tower for $l_K: K\rightarrow \mathrm{BS}\Gamma_{2n}$ and $\bold{MT}\eta$ is the Madsen-Tillman spectrum associated to $\eta$.
\end{thm}
To obtain \Cref{limit homology} from \Cref{universalend},  we first need to choose the following universal $\nu$-end.
\begin{prop}Recall that the point foliation on $\ttW_{1,2}$ gives a $\nu$-structure. Let $( K|_{[i,i+1]}, l_{K|_{[i,i+1]}})$ be $\ttW_{1,2}$ with the chosen $\nu$-structure. Then the pair $(K, l_K)$ is a universal $\nu$-end. \end{prop}
\begin{proof}Note that since $\tW_{1,2}$ is $(n-1)$-connected, it has a  $\theta^n$-structure where $\theta^n:  \mathrm{BGL}_{2n}(\bR)\langle n\rangle\rightarrow \mathrm{BGL}_{2n}(\bR)$.  Using \cite[ Addendum 1.9]{galatius2012stable} it is easy to show that $K$ with its $\theta$-structure is a universal $\theta$-end. Now recall that the map $\nu: \mathrm{BS}\Gamma_{2n}\rightarrow \mathrm{BGL}_{2n}^+(\bR)$ is $(2n+2)$-connected (\cite[Remark 1]{haefliger1971homotopy}). Since $K$ is a universal $\theta$-end, the conditions in \Cref{defn1} hold for $\mathrm{BGL}_{2n}^+(\bR)$, and therefore they also hold for $\mathrm{BS}\Gamma_{2n}$.
\end{proof}
Now given the connectivity of the map $\nu$, it is easy to see that $\nu^n: \mathrm{B}\Gamma_{2n}\langle n\rangle\rightarrow \mathrm{BS}\Gamma_{2n}\rightarrow \mathrm{BGL}_{2n}^+(\bR)$ is the $n$-th stage of the Moore-Postnikov tower for $l_K: K\rightarrow \mathrm{BS}\Gamma_{2n}$.

\vspace{2mm}
\subsection{Stable splitting after the $p$-adic completion}
In order to understand the effect of the map $\BdDiff(\W,\partial)\rightarrow \BDiff(\W,\partial)$ on the level of cohomology in the stable range, we will study the  map
     \[
 \begin{tikzpicture}[node distance=3.2cm, auto]
  \node (A) {$\Omega_0^{\infty}\bold{MT}\nu^n$};
  \node (B) [right of=A] {$\Omega_0^{\infty}\bold{MT}\theta^n,$};
  \draw [->] (A) to node {$$}(B);
\end{tikzpicture}
\]
and we shall prove below that this map is a split surjection after $p$-adic completion (see \cite[Part 3]{may2011more} for a definition of $p$-adic completion of spaces). 
\begin{thm}\label{split}
 The  natural map 
     \[
 \begin{tikzpicture}[node distance=3.2cm, auto]
  \node (A) {$\Omega_0^{\infty}\bold{MT}\nu^n$};
  \node (B) [right of=A] {$\Omega_0^{\infty}\bold{MT}\theta^n$};
  \draw [->] (A) to node {$$}(B);
\end{tikzpicture}
\]
is a split surjection after $p$-adic completion for all primes $p$.
\end{thm}
\begin{cor}
The induced map by the identity homomorphism
     \[
 \begin{tikzpicture}[node distance=4.5cm, auto]
  \node (A) {$H^*(\BDiff(\WW,\partial);\bF_p)$};
  \node (B) [right of=A] {$H^*(\BdDiff(\WW,\partial);\bF_p)$};
  \draw [right hook->] (A) to node {$$}(B);
\end{tikzpicture}
\]
is injective provided that $*< (g-2)/2$.
\end{cor}
Recall that the map 
  \[
 \begin{tikzpicture}[node distance=2.2cm, auto]
  \node (A) {$\mathrm{BS}\Gamma_{2n}$};
  \node (B) [right of=A] {$\mathrm{BGL}_{2n}^+(\bR)$};
  \draw [->] (A) to node {$\nu$}(B);
\end{tikzpicture}
\]
is induced by the continuous map of topological groupoids $\hat{\nu}:\mathrm{S}\Gamma_{2n}\rightarrow \mathrm{GL}_{2n}(\bR)^+$, where $\hat{\nu}$ sends a germ $f\in \mathrm{S}\Gamma_{2n}$, to its derivative $df$ evaluated at the source of the germ $f$. Furthermore, there is an obvious map $\tilde{\iota}: \mathrm{SO}(2n)^{\delta}\rightarrow \mathrm{S}\Gamma_{2n}$, which assigns to a matrix its germ as a diffeomorphism of $\bR^{2n}$ at $0$. Note that the image of the composite $\hat{\nu}\circ \tilde{\iota}$ is $\mathrm{SO}(2n)$. Thus, we have

 \[
 \begin{tikzpicture}[node distance=2.2cm, auto]
 \node (C) {$\mathrm{BSO}(2n)^{\delta}$};
  \node (A) [right of=C] {$\mathrm{BS}\Gamma_{2n}$};
  \node (B) [right of=A, node distance=3cm] {$\mathrm{BGL}_{2n}^+(\bR)\simeq \mathrm{BSO}(2n)$};
  \draw [->] (A) to node {$\nu$}(B);
  \draw [->] (C) to node {$\iota$}(A);
\end{tikzpicture}
\]
where $\nu\circ\iota$ is homotopic to the map induced by the identity from $\mathrm{SO}(2n)^{\delta}$ to $\mathrm{SO}(2n)$. Hence, we have the following maps between Thom spectra
\[
 \begin{tikzpicture}[node distance=3.2cm, auto]
 \node (C) {$(\mathrm{BSO}(2n)^{\delta})^{-\nu\circ \iota}$};
  \node (A) [right of=C] {$\mathrm{BS}\Gamma_{2n}^{-\nu}$};
  \node (B) [right of=A, node distance= 2.6cm] {$ \mathrm{BSO}(2n)^{-\gamma}.$};
  \draw [->] (A) to node {$\nu'$}(B);
  \draw [->] (C) to node {$\iota'$}(A);
\end{tikzpicture}
\] 

The Milnor conjecture \cite{milnor1983homology} says that for a Lie group $G$, the classifying space ${\mathrm B}G$ and ${\mathrm B}G^{\delta}$ are $p$-adically equivalent. If the Milnor conjecture were known for $\mathrm{SO}(2n)$, the proof of the theorem would be much shorter, because then $ \mathrm{BSO}(2n)$ and $\mathrm{BSO}(2n)^{\delta}$ would be  equivalent after $p$-adic completion and by Thom isomorphism so were $ \mathrm{BSO}(2n)^{-\gamma}$ and $(\mathrm{BSO}(2n)^{\delta})^{-\nu\circ \iota}$. Hence, this equivalence implies $\nu'$ splits after $p$-completion. To prove the theorem, we first show that  $\nu'$ has a section after $p$-adic completion. Because Milnor's conjecture seems to be unknown for real Lie groups, we give a transfer argument to prove that $\nu'$ admits a section after $p$-completion.
 \begin{lem}\label{BO(2n)}
 The  map $\nu'\circ\iota'$  splits after $p$-completion
  \[ 
 \begin{tikzpicture}[node distance=3.2cm, auto]
  \node (A) {$(\mathrm{BSO}(2n)^{\delta})^{-\nu\circ \iota}$};
  \node (B) [right of=A] {$\mathrm{BSO}(2n)^{-\gamma}$};
  \draw [->] (A) to node {$$}(B);
\end{tikzpicture}
\]
i.e. it admits a section after $p$-completion.
 \end{lem}
 \begin{proof}
 We denote the normalizer of the maximal torus in $\mathrm{SO}(2n)$ by $ \mathrm{N}(T)$. Consider the following homotopy commutative diagram
 \[
 \begin{tikzpicture}[node distance=2cm, auto]
  \node (A) {$\mathrm{BSO}(2n)^{\delta}$};
  \node (B) [below of=A]{$\mathrm{B}(\mathrm{N}(T)^{\delta})$};
  \node (C) [right of=A, node distance=3.5cm]{$ \mathrm{BSO}(2n)$};
  \node (D) [right of=B, node distance=3.5cm]{$ {\mathrm B}(\mathrm{N}(T))$};
  \draw [->] (A) to node {} (C);
  \draw [->] (B) to node {} (D);
  \draw [->] (B) to node {$i^{\delta}$} (A);
  \draw [->] (D) to node {$i$} (C);
 \end{tikzpicture}
 \]
where $i$  is induced by injection of $\mathrm{N}(T)$ into $\mathrm{SO}(2n)$. Similarly, the map $i^{\delta}$ is defined between the same Lie groups made discrete. Similar to \cite[Section 3]{milnor1983homology}, there is a canonical fibration
  \[
 \begin{tikzpicture}[node distance=3.5cm, auto]
  \node (A) {$\mathrm{BSO}(2n),$};
  \node (B) [left of=A] {$i: {\mathrm B}(\mathrm{N}(T))$};
  \draw [->] (B) to node {$$}(A);
\end{tikzpicture}
\]
whose fiber is $\mathrm{SO}(2n)/\mathrm{N}(T)$. Using the Becker-Gottlieb transfer and the theorem of Hopf and Samelson which says that the Euler number of $\mathrm{SO}(2n)/\mathrm{N}(T)$ is one (\cite{MR0006546}), we deduce that the map $i$ induces a surjection on homology  with any coefficient group. Using the Thom isomorphism and the twisted Becker-Gottlieb transfer \cite[Section 3.2]{kashiwabara2014splitting} for the induced map between Thom spectra
  \[
 \begin{tikzpicture}[node distance=3.5cm, auto]
  \node (A) {$ \mathrm{BSO}(2n)^{-\gamma},$};
  \node (B) [left of=A] {$i': {\mathrm B}(\mathrm{N}(T))^{-i}$};
  \draw [->] (B) to node {$$}(A);
\end{tikzpicture}
\]
we have a transfer map $\tau: \mathrm{BSO}(2n)^{-\gamma}\to {\mathrm B}(\mathrm{N}(T))^{-i}$ so that $i'\circ \tau$ induces a homology isomorphism for all coefficient groups. Therefore the transfer map $\tau$ splits off $\mathrm{BSO}(2n)^{-\gamma}$ from the spectrum ${\mathrm B}(\mathrm{N}(T))^{-i}$. Now for a prime $p$, we consider the following diagram between $p$-completed spectra
\begin{equation}
\begin{gathered}
 \begin{tikzpicture}[node distance=2cm, auto]
  \node (A) {$((\mathrm{BSO}(2n)^{\delta})^{-\nu\circ \iota})^{\wedge}_p$};
  \node (B) [below of=A]{$(\mathrm{B}(\mathrm{N}(T)^{\delta})^{-\nu\circ \iota\circ i^{\delta}})^{\wedge}_p$};
  \node (C) [right of=A, node distance=5.2cm]{$ (\mathrm{BSO}(2n)^{-\gamma})^{\wedge}_p$};
  \node (D) [right of=B, node distance=5.2cm]{$( {\mathrm B}(\mathrm{N}(T))^{-i})^{\wedge}_p.$};
  \draw [->] (A) to node {$(\nu'\circ\iota')^{\wedge}_p$} (C);
  \draw [->] (B) to node {$\simeq$} (D);
  \draw [->] (B) to node {$$} (A);
  \draw [->] (C) to node {$\tau^{\wedge}_p$} (D);
 \end{tikzpicture}
 \end{gathered}
 \end{equation}

 If we show the bottom horizontal map is a weak equivalence, we will obtain a section for $\nu'\circ \iota'$ after $p$-adic completion. In order to prove this weak equivalence, it is sufficient to prove that the bottom horizontal map induces an isomorphism on mod $p$ homology  (see \cite[Theorem 11.1.2]{may2011more}). Using the Thom isomorphism, we only need to show that the middle map in the following diagram induces an isomorphism on mod $p$ homology
 
   \[
 \begin{tikzpicture}[node distance=4.1cm, auto]
  \node (A) {${\mathrm B}T^{\delta}$};
  \node (B) [below of=A, node distance=1.8cm] {$\mathrm{B}(\mathrm{N}(T)^{\delta})$};
  \node (C) [below of= B, node distance=1.8cm ] {${\mathrm B}\mathrm{W}$};  
  \node (D) [right of= A] {${\mathrm B}T$};
    \node (E) [right of= B] {$\mathrm{B}(\mathrm{N}(T))$};
  \node (F) [right of= C ] {${\mathrm B}\mathrm{W}$};  
   \draw [->] (A) to node {\tiny{$p$-adic equivalence}}(D);
  \draw [->] (B) to node {$$}(E);
  \draw [->] (C) to node {$=$}(F);
  \draw [->] (A) to node {$$}(B);
  \draw [->] (B) to node {$$} (C);
 \draw [->] (D) to node {$$} (E);
  \draw [->] (E) to node {$$} (F);

\end{tikzpicture}
\]
where $\mathrm{W}$ is the Weyl group of $\mathrm{SO}(2n)$. It is a special case of \cite[Lemma 3]{milnor1983homology} that  $({\mathrm B}T^{\delta})^{\wedge}_p\simeq ({\mathrm B}T)^{\wedge}_p$. Hence, the top horizontal map is a mod $p$ homology isomorphism. The actions of Weyl group $\mathrm{W}$ on the cohomology of fibers with $\bF_p$ coefficients are the same so by comparing the Leray-Serre spectral sequences, we see the middle map is a mod $p$ homology isomorphism. 
 \end{proof}
\begin{proof}[Proof of \Cref{split}]

 First, we show that on the level of spectra, the map $\bold{MT}\nu^n\rightarrow \bold{MT}\theta^n$,  admits a section after the $p$-adic completion. Having splitting on the level of spectra, we then show that the splitting of $\Omega_0^{\infty}\bold{MT}\nu^n$ formally follows from the properties of the $p$-completion. Consider the following diagram of spectra

\[
 \begin{tikzpicture}[node distance=2cm, auto]
  \node (A) {$\bold{MT}\nu^n$};
  \node (B) [below of=A]{$\mathrm{BS}\Gamma_{2n}^{-\nu}$};
  \node (C) [right of=A, node distance=3cm]{$ \bold{MT}\theta^n$};
  \node (D) [right of=B, node distance=3cm]{$ \mathrm{BSO}(2n) ^{-\gamma}.$};
  \draw [->] (A) to node {$\nu''$} (C);
  \draw [->] (B) to node {$\nu'$} (D);
  \draw [->] (A) to node {} (B);
  \draw [->] (C) to node {} (D);
 \end{tikzpicture}
 \]
In \Cref{BO(2n)}, we proved that the composite  
\begin{equation}\label{1}
(\mathrm{BSO}(2n)^{\delta})^{-\nu\circ \iota}\longrightarrow \mathrm{BS}\Gamma_{2n}^{-\nu}\longrightarrow\mathrm{BSO}(2n)^{-\gamma},
\end{equation}
has a section after $p$-adic completion. Hence $\nu'$ also admits a section after $p$-adic completion.

$\textbf{Step 1:}$ We  prove that the map $\nu''$ also has a section after $p$-adic completion. Let $T$ be the maximal torus in $\mathrm{SO}(2n)$ and $\text{N}(T)$ be the normalizer of the torus in $\mathrm{SO}(2n)$. By the same arguments in the proof of \Cref{BO(2n)}, we have the following commutative diagram
\[
 \begin{tikzpicture}[node distance=2cm, auto]
  \node (A) {$(\mathrm{B}(\mathrm{N}(T)^{\delta}))^{\wedge}_p\simeq(\mathrm{B}\text{N}(T))^{\wedge}_p$};
  \node (B) [right of=A, above of=A, node distance=1.7cm]{$(\mathrm{BS}\Gamma_{2n})^{\wedge}_p$};
  \node (C) [right of=A, node distance=4cm]{$ (\mathrm{BSO}(2n))^{\wedge}_p.$};
  \draw [->] (A) to node {$$} (C);
  \draw [->] (B) to node {$$} (C);
  \draw [->] (A) to node {} (B);
 \end{tikzpicture}
\]
Note that there exists the  Becker-Gottlieb transfer for the bottom horizontal map even before $p$-completion. Let   $Y$ and $Y'$ be the homotopy pullbacks in the following diagram
\begin{equation}\label{Y}
  \begin{gathered}
 \begin{tikzpicture}[node distance=2cm, auto]
  \node (A) {$Y$};
  \node (B) [below of=A]{$\mathrm{B}\text{N}(T)$};
  \node (C) [right of=A, node distance=3cm]{$ \mathrm{BSO}(2n)\langle n\rangle$};
  \node (D) [right of=B, node distance=3cm]{$ \mathrm{BSO}(2n).$};
  \node (E) [left of=A]{$Y'$};
  \node (F) [below of=E]{$\mathrm{B}\text{N}(T)^{\delta}$};
  \draw [->] (A) to node {$$} (C);
  \draw [->] (B) to node {$$} (D);
  \draw [->] (A) to node {} (B);
  \draw [->] (C) to node {} (D);
  \draw [->] (E) to node {} (A);
  \draw [->] (E) to node {} (F);
  \draw [->] (F) to node {} (B);
 \end{tikzpicture}
   \end{gathered}
 \end{equation}
The map from $Y$ to $ \mathrm{BSO}(2n)\langle n\rangle$ also admits a Becker-Gottlieb transfer because its fiber is $\mathrm{SO}(2n)/\mathrm{N}(T)$. We show that there exists a map $Y'\rightarrow \mathrm{BS}\Gamma_{2n}\langle n\rangle$, making the following diagram commutative
 \[
 \begin{tikzpicture}[node distance=2cm, auto]
  \node (A) {$Y'$};
  \node (B) [right of=A, above of=A, node distance=1.7cm]{$\mathrm{BS}\Gamma_{2n}\langle n\rangle$};
  \node (C) [right of=A, node distance=4cm]{$ \mathrm{BSO}(2n)\langle n\rangle.$};
  \draw [->] (A) to node {$$} (C);
  \draw [->] (B) to node {$$} (C);
  \draw [->, dotted ] (A) to node {} (B);
 \end{tikzpicture}
\]
Consider the following homotopy commutative diagram
\begin{equation}\label{c1}
\begin{gathered}
 \begin{tikzpicture}[node distance=2cm, auto]
  \node (A) {$\mathrm{BS}\Gamma_{2n}\langle n\rangle$};
  \node (B) [below of=A]{$\mathrm{BS}\Gamma_{2n}$};
  \node (C) [right of=A, node distance=3cm]{$ \mathrm{BSO}(2n)\langle n\rangle $};
  \node (D) [right of=B, node distance=3cm]{$ \mathrm{BSO}(2n)$};
  \node (E) [left of=A, above of=A, node distance=1.1cm]{$Y'$};
  \draw [->] (A) to node {$$} (C);
  \draw [->] (B) to node {$$} (D);
  \draw [->] (A) to node {} (B);
  \draw [->] (C) to node {} (D);
  \draw [->, bend left=15] (E) to node {} (C);
  \draw [->, bend right] (E) to node {} (B);  
 \draw [->, dotted] (E) to node {} (A);
 \end{tikzpicture}
 \end{gathered}
 \end{equation}
where the left bent arrow is given by the composition $Y'\rightarrow \mathrm{B}\text{N}(T)^{\delta}\rightarrow \mathrm{BS}\Gamma_{2n}$. By Haefliger's theorem \cite[Remark 1]{haefliger1971homotopy}, the square is a homotopy pullback square, so the dotted arrow exists up to homotopy. 

Let us, with abuse of notation, denote by $\gamma$ the pullbacks of tautological bundle over $Y$ and $Y'$. And we denote the Thom spectrum of $-\gamma$ over $Y$ and $Y'$ respectively by $Y^{-\gamma}$ and $Y'^{-\gamma}$. 

Since the right vertical map in diagram \ref{Y} is between simply connected spaces, if we take $p$-completion of diagram \ref{Y}, all homotopy pullback squares remain homotopy pullback squares. Given that $(\text{N}(T))^{\wedge}_p\simeq (\text{N}(T)^{\delta})^{\wedge}_p$ and using Thom isomorphism, we have $(Y^{-\gamma})^{\wedge}_p\simeq (Y'^{-\gamma})^{\wedge}_p$. Since the fiber of the map
\[
Y\to \mathrm{BSO}(2n)\langle n\rangle
\]
is a manifold, we have a transfer map $\bold{MT}\theta^n\to Y^{-\gamma}$.
Hence, using this transfer map in the diagram \ref{c1}, we obtain 
\[
 \begin{tikzpicture}[node distance=3.5cm, auto]
  \node (A) {$ (\bold{MT}\theta^n)^{\wedge}_p$};
  \node (B) [right of=A] {$(Y^{-\gamma})^{\wedge}_p\simeq (Y'^{-\gamma})^{\wedge}_p$};
  \node (C) [right of=B, node distance=3.5cm] {$(\bold{MT}\nu^n)^{\wedge}_p,$};
  \draw [->] (A) to node {$$}(B);
  \draw [->] (B) to node {} (C);
\end{tikzpicture}
\]
which provides a section for $\nu''$ after $p$-adic completion.

 $\textbf{Step 2:}$ The last step is to use this section on the level of spectra and prove that it induces a section on the corresponding  infinite loop spaces, i.e. we want to show that the following map has a section
  \[
 \begin{tikzpicture}[node distance=5.4cm, auto]
  \node (B) [right of=A] {$ (\Omega^{\infty}_0\mathrm{BSO}(2n)^{-\gamma})^{\wedge}_p,$};
  \node (A) {$(\Omega^{\infty}_0(\mathrm{BSO}(2n)^{\delta})^{-\nu\circ \iota})^{\wedge}_p$};
  \draw [->] (A.355) to node [swap] {$\nu'\circ \iota'$}(B.185);
  \draw [->,dashed] (B.173) to node [swap] {$s$}(A.5);
\end{tikzpicture}
\]
but this is a consequence of the fact that if $X$ is a spectrum, then $\Omega^{\infty}_0(X^{\wedge}_p)$ is a $p$-completed space and it is weakly equivalent to $(\Omega^{\infty}_0(X))^{\wedge}_p$. Note that homotopy groups of $\Omega^{\infty}_0(X^{\wedge}_p)$ are the positive homotopy groups of $X^{\wedge}_p$ and these groups can be computed by  the following exact sequence
 \[
\begin{tikzpicture}[node distance=3cm, auto]
\node (E) [left of=A, node distance=2.3cm] {$0$};
  \node (A) {$\text{Ext}(\bZ/p^{\infty}, \pi_*(X))$};
  \node (B) [right of=A] {$\pi_*(X^{\wedge}_p)$};
  \node (C) [right of=B] {$\text{Hom}(\bZ/p^{\infty},\pi_{*-1}(X))$};
  \node (K) [right of=C, node distance=2.5cm] {$0,$};
  \draw [->] (A) to node {$$}(B);
  \draw [->] (B) to node {$$}(C);
  \draw [->] (E) to node {} (A);
  \draw [->] (C) to node {} (K);
\end{tikzpicture}
 \] 
where $\bZ/p^{\infty}= \bZ[1/p]/\bZ$. Since the two outer terms are $p$-completed groups, so are the homotopy groups $\pi_*(X^{\wedge}_p)$. Hence, the fact that the homotopy groups of $\Omega^{\infty}_0(X^{\wedge}_p)$ are $p$-completed groups, \cite[Theorem 11.1.1]{may2011more} implies that $\Omega^{\infty}_0(X^{\wedge}_p)$ is a $p$-completed space. Thus, by the universal property of $p$-completion, there exists a map $(\Omega^{\infty}_0(X))^{\wedge}_p\rightarrow \Omega^{\infty}_0(X^{\wedge}_p)$. Given that homotopy groups of $(\Omega^{\infty}_0(X))^{\wedge}_p$ can be obtained by the same exact sequence, we deduce that it has the same homotopy groups as $ \Omega^{\infty}_0(X^{\wedge}_p)$,  hence $(\Omega^{\infty}_0(X))^{\wedge}_p\simeq \Omega^{\infty}_0(X^{\wedge}_p)$. This weak equivalence finishes the proof by providing the following section
 \[
\begin{tikzpicture}[node distance=1.8cm, auto]
  \node (A) {$ \Omega^{\infty}_0((\mathrm{BSO}(2n)^{-\gamma})^{\wedge}_p)$};
  \node (C) [below of=A]{$ (\Omega^{\infty}_0\mathrm{BSO}(2n)^{-\gamma})^{\wedge}_p$};
  \node (B) [right of=A, node distance=5cm] {$\Omega^{\infty}_0((\mathrm{BSO}(2n)^{\delta})^{-\nu\circ \iota})^{\wedge}_p) $};
  \node (D) [right of=C,node distance=5cm] {$(\Omega^{\infty}_0(\mathrm{BSO}(2n)^{\delta})^{-\nu\circ \iota})^{\wedge}_p.$};
  \draw [->] (A) to node {$$}(B);
  \draw [->] (C) to node {$\simeq$}(A);
  \draw [->] (D) to node {$\simeq$} (B);
  \draw [->,dashed] (C) to node {$$} (D);
\end{tikzpicture}
 \] 

\end{proof}
\begin{cor}\label{nontorsion}
The induced  map  \[(\Omega^{\infty}\nu^n)^*: H^*(\Omega_0^{\infty}\bold{MT}\theta^n;\bZ)\rightarrow H^*(\Omega_0^{\infty}\bold{MT}\nu^n;\bZ)\]  is injective. 
\end{cor}
\begin{proof}
Suppose we have {\Small $(\Omega^{\infty}\nu^n)^*(a)=0$} for  a non-trivial element $a\in H^*(\Omega_0^{\infty}\bold{MT}\theta^n;\bZ)$. Consider the following commutative diagram between the Bokstein exact sequences
 \[
 \begin{tikzpicture}[node distance=5cm, auto]
  \node (A) {$H^*(\Omega_0^{\infty}\bold{MT}\theta^n;\bZ)$};
  \node (B) [right of=A] {$H^*(\Omega_0^{\infty}\bold{MT}\nu^n;\bZ)$};
  \node (C) [node distance=2cm, below of=A] {$H^*(\Omega_0^{\infty}\bold{MT}\theta^n;\bF_p)$};  
  \node (D) [node distance=2cm, below of=B] {$H^*(\Omega_0^{\infty}\bold{MT}\nu^n;\bF_p),$};
  \node (E) [node distance=2cm, above of=A]{$H^*(\Omega_0^{\infty}\bold{MT}\theta^n;\bZ)$};
    \node (F) [node distance=2cm, above of=B] {$H^*(\Omega_0^{\infty}\bold{MT}\nu^n;\bZ)$};
  \draw [->] (E) to node {$$} (F);
  \draw [->] (E) to node {$\times p$} (A);
  \draw [->] (F) to node {$\times p$} (B);
  \draw[right hook->] (C) to node  {}(D);
  \draw [->] (A) to node {$i$}(C);
  \draw [->] (A) to node {$$} (B);
  \draw [->] (B) to node {$i'$} (D);
\end{tikzpicture}
\]
thus $a\in pH^*(\Omega_0^{\infty}\bold{MT}\theta^n;\bZ)$ for all $p$. Since $H^*(\Omega_0^{\infty}\bold{MT}\theta^n;\bZ)$ is finitely generated in each degree (\cite[Theorem 1.1]{galatius2012stable}), we deduce that $a=0$. 

\end{proof} 
\section{Remarks on characteristic classes of flat $\W$-bundles}\label{remarks}
 The goal of this section is two fold. On the one hand  we study the image of the map
 \[
\iota^*: H^*(\BDiff(\W,\partial);\bZ)\to H^*(\BdDiff(\W,\partial);\bZ).
  \]
 On the other hand we use \Cref{limit homology} and what is known about $H^*(\mathrm{BS}\Gamma_{2n};\bZ)$ to detect non-trivial cohomology classes of $\BdDiff(\W,\partial)$ that are not in the image of $\iota^*$. 
\subsection{On generalized MMM-classes for flat $\W$-bundles}\label{sec6} There are generalized MMM classes $\kappa_c$ in $H^k(\BDiff(\W))$ associated to each $c\in H^{k+2n}(\mathrm{BSO}(2n))$ defined as follows. Consider the universal $\tW_{g}$-bundle
 \[
\begin{tikzpicture}[node distance=1.4cm, auto]
  \node (A) {$\tW_{g}$};
  \node (B) [right of=A] {$E$};
  \node (C) [right of=B, node distance=2.1cm] {$\BDiff(\tW_{g}).$};
  \draw [->] (A) to node {$$}(B);
  \draw [->] (B) to node {$\pi$}(C);
\end{tikzpicture}
 \] 
 The vertical tangent bundle $T_{\pi}E\rightarrow E$ is a $2n$-dimensional bundle over $E$ which restricts to the tangent bundle of each fiber. Thus, to any class $c\in H^{2n+k}(\mathrm{BSO}(2n))$, we can associate a class $c(T_{\pi}E)\in H^{2n+k}(E)$. The fiber is a closed compact manifold, so we can integrate this class along the fiber  and obtain a cohomology class in the cohomology of the base 
 \begin{equation*}
 \kappa_c=\pi_!c(T_{\pi}E)\in H^k(\BDiff(\tW_{g})).
 \end{equation*}
 We can pull back $\kappa_c$ via the natural injection of $\Diff(\W,\partial)\hookrightarrow \Diff(\tW_{g})$ to obtain a cohomology class in $H^k(\BDiff(\W,\partial);\bZ)$. We can further pull it back to $H^k(\BdDiff(\W,\partial);\bZ)$ and denote it by  $\kappa_c^\delta$ .
 
 The following theorem is proved in \cite[Corollary 1.8]{galatius2014homological},
 \begin{thm}\label{rationalcohomology}
 Let $n>2$ and let $\mathcal{B}\subset H^*(\mathrm{BSO}(2n);\bQ)$ be the set of monomials in the classes $e,p_{n-1},\dots,p_{\left \lceil{n+1}/{4} \right\rceil }$, of degrees larger than $2n$ where $e$ is the Euler class and $p_i$ denotes the $i$-th Pontryagin class. Then, the induced map
 \[
 \bQ[\kappa_c | c\in \mathcal{B}]\rightarrow H^*(\BDiff(\WW,\partial);\bQ)
 \]
 is an isomorphism in the range $*\leq (g-3)/2$.
 \end{thm}
 
  One  consequence of \Cref{nontorsion} and \Cref{rationalcohomology} is the following corollary. 
 \begin{cor}\label{cor4}
For all $c\in \mathcal{B} $, we have a map
\[
\begin{tikzpicture}[node distance=3.7cm, auto]
  \node (A) {$\bZ[\kappa_c^{\delta}|c\in \mathcal{B}]$};
  \node (B) [right of=A] {$H^*(\BdDiff(\ttW_{\infty,1},\partial);\bZ)$};
  \draw [right hook->] (A) to node {$$}(B);
\end{tikzpicture}
\] 
which is injective.
\end{cor}
This corollary implies that those monomials of $\kappa_c^{\delta}$ for $c\in \mathcal{B}$ that lie in the stable cohomology of $H^*(\BdDiff(\W,\partial)$ are nontrivial. Using a Lie group action on $\tW_{g}$, one can show that there are nontrivial $\kappa_c^{\delta}$ in the unstable range of $H^*(\BdDiff(\tW_{g});\bZ)$.
\begin{thm}
For  $g>1$ and $i<n$, the class $\kappa_{ep_i}^{\delta}\in H^{4i}(\BdDiff(\text{\textnormal{W}}_g);\bZ)$ is nontrivial.
\end{thm}
\begin{proof}
Galatius, Grigoriev and Randal-Williams proved in \cite[Theorem 4.1 (ii)]{galatius2015tautological} that there exists an $SO(n)\times SO(n)$-action on $\tW_g$ where for all $i\in \{1,2,\dots, n-1\}$ and $g>1$, the class $\kappa_{ep_{i}}$ is not in the kernel of the induced map 
\[
H^*(\BDiff(\tW_g);\bQ)\to H^*(\mathrm{BSO}(n)\times \mathrm{BSO}(n);\bQ).
\] 
Hence, we have the following commutative diagram
 \[
 \begin{tikzpicture}[node distance=5cm, auto]
  \node (A) {$H^*(\BDiff(\tW_g);\bZ)$};
  \node (B) [right of=A] {$H^*(\mathrm{BSO}(n)\times \mathrm{BSO}(n);\bZ)$};
  \node (C) [node distance=2cm, below of=A] {$H^*(\BdDiff(\tW_g);\bZ)$};  
  \node (D) [node distance=2cm, below of=B] {$H^*(\mathrm{BSO}(n)^{\delta}\times \mathrm{BSO}(n)^{\delta};\bZ),$};
  \draw[->] (C) to node  {}(D);
  \draw [->] (A) to node {$i$}(C);
  \draw [->] (A) to node {$$} (B);
  \draw [->] (B) to node {$i'$} (D);
\end{tikzpicture}
\]
where by the result of Milnor \cite[Corollary 1]{milnor1983homology}, we know $i'$ is injective. Injectivity of $i'$ implies that the image of $\kappa_{ep_{i}}$ in $H^*(\BdDiff(\tW_g);\bZ)$ is nonzero for all $g>1$.
\end{proof}
Now we want to show that MMM-classes of degrees larger than $4n$ vanish in stable rational cohomology of $\BdDiff(\tW_g)$, which implies that the following map does not admit a section
 \[
\begin{tikzpicture}[node distance=2.9cm, auto]
  \node (A) {$\BdDiff(\tW_g)$};
  \node (B) [right of=A] {$\BDiff(\tW_g).$};
  \draw [->] (A) to node {$$}(B);
\end{tikzpicture}
\]
\begin{prop}\label{vanishing}
If $c$ is a monomial generated by $e, p_{n-1},\cdots, p_{1}$ of degrees larger than $6n$, then $\kappa_c^{\delta}$'s vanish  in $H^*(\BdDiff(\tW_g);\bQ)$.
\end{prop}
\begin{proof}
We need to prove that for any flat $\tW_g$-bundle, $E \xrightarrow{\pi} M$, its $\kappa_c^\delta$  vanishes as long as $c>6n$.  Recall that the Bott vanishing theorem \cite{bott1970topological} says for a foliation $\mathcal{F}$ on $E$ of codimension $q$, we have
\[
\text{Pont}^{>2q}(\nu \mathcal{F})=0
\]
where  $\text{Pont}^{>2q}(\nu \mathcal{F})$ is a ring generated by monomials of Pontryagin classes of the normal bundle of $\mathcal{F}$ of degree larger than $2q$. Any flat $\tW_g$-bundle structure on $E$ gives a foliation of codimension $2n$ such that the vertical tangent bundle is the normal bundle of the of the given foliation. Suppose
 \[c=e(\text{T}_{\pi}E)^{a}p_{i_1}(\text{T}_{\pi}E)^{a_1}\cdots p_{i_k}(\text{T}_{\pi}E)^{a_k}\]
  if $a\leq 1$ then we have $\sum 4i_ja_j>4n$; Since by the Bott vanishing theorem, the class  $p_{i_1}(\text{T}_{\pi}E)^{a_1}\cdots p_{i_k}(\text{T}_{\pi}E)^{a_k}$ has to vanish so does $c$. If $a>1$ then we have $4n\left \lfloor {a}/{2}\right \rfloor + \sum 4i_ja_j>4n$, again by the Bott vanishing theorem the class \[p_n(\text{T}_{\pi}E)^{\left \lfloor {a}/{2}\right \rfloor}p_{i_1}(\text{T}_{\pi}E)^{a_1}\cdots p_{i_k}(\text{T}_{\pi}E)^{a_k}\] has to vanish, so does $c$.
\end{proof}

 \begin{rem}\Cref{vanishing} suggests that there should be $\bQ\slash \bZ$-characteristic classes. One can use Cheeger-Simons theory to lift certain MMM-classes to $\bR\slash\bZ$-classes. Let $\tW_g\to E\to M$ be a flat manifold bundle. Let $c\in  H^*(E; \bR)$ be a monomial consisting of Pontryagin classes of the vertical tangent bundle. The Bott vanishing implies that $c$ is zero if $\text{deg}(c)>4n$. Cheeger-Simons theory (see \cite[corollary 2.4]{cheeger1985differential}) associates  a natural secondary characteristic class  to $c$,  denoted by $\hat{c}$ that  lives in $ H^{\text{deg}(c)-1}(E;\bR/\bZ)$ and the Bockstein map
 \[
 \beta: H^{\text{deg}(c)-1}(E;\bR/\bZ)\to H^{\text{deg}(c)}(E;\bZ)
 \]
 sends $\hat{c}$ to $-c$. Therefore  to every monomial of Pontryagin classes in $H^{*}(\mathrm{BSO}(2n); \bZ)$ denoted by  $c$ whose degree is larger than $4n$, one can associate a universal class $\hat{c}\in  H^{\text{deg}(c)-1}(\mathrm{BS}\Gamma_{2n};\bR/\bZ)$. One can also pullback this class to ${\mathrm B}\Gamma_{2n}\langle n\rangle$ and use the Thom isomorphism to obtain a class in $H^{\text{deg}(c)-2n-1}(\bold{MT}\nu^n;\bR/\bZ)$. Let  $\widehat{\kappa_c}$  denote the image of this class under the cohomology suspension map

 \[
\begin{tikzpicture}[node distance=5.9cm, auto]
  \node (C) {$\sigma^*: H^{\text{deg}(c)-2n-1}(\bold{MT}\nu^n;\bR/\bZ)$};
  \node (D) [right of=C] {$H^{\text{deg}(c)-2n-1}(\Omega^{\infty}_0\bold{MT}\nu^n;\bR/\bZ),$};
  \draw[->] (C) to node {$$}(D);
\end{tikzpicture}
\]
if $\text{deg}(c)-2n-1$ lies in the stable range,  we have \[H^{\text{deg}(c)-2n-1}(\Omega^{\infty}_0\bold{MT}\nu^n;\bR/\bZ)=H^{\text{deg}(c)-2n-1}(\BdDiff(\W,\partial);\bR/\bZ).\]
Using the naturality of these classes, similar to \cite[corollary 2.4]{cheeger1985differential}, one can see that  $ \widehat{\kappa_c}$ maps to $-\kappa_c^{\delta}$ under the Bockstein map
\[
\begin{tikzpicture}[node distance=5.9cm, auto]
  \node (C) {$H^{\text{deg}(c)-2n-1}(\BdDiff(\W,\partial);\bR/\bZ)$};
  \node (D) [right of=C] {$H^{\text{deg}(c)-2n}(\BdDiff(\W,\partial);\bZ).$};
  \draw[->] (C) to node {$\beta$}(D);
\end{tikzpicture}
\]

 By the virtue of \Cref{nontorsion}, we know that those $\kappa_c^{\delta}$'s that live in the stable range are non-torison classes in $H^*(\BdDiff(\W,\partial);\bZ)$; thus, corresponding $\widehat{\kappa_c}$'s are nontrivial and non-torsion classes. They induce a map
 \[
\begin{tikzpicture}[node distance=3.9cm, auto]
  \node (C) {$H_{\text{deg}(c)-2n-1}(\BdDiff(\W,\partial);\bZ)$};
  \node (D) [right of=C] {$\bR/\bZ.$};
  \draw[->] (C) to node {$\widehat{\kappa_c}$}(D);
\end{tikzpicture}
\]
Hence for those $c$ with $\text{deg}(c)>4n$, we have $H_{\text{deg}(c)-2n-1}(\BdDiff(\W,\partial);\bZ)$ is nontrivial. 
\end{rem}
\begin{thm}
For $c\in \mathcal{B}$ and $\textit{deg}(c)>6n$, let $k=\textit{deg}(c)-2n-1$. Then for such $k$, the group $H_k(\BdDiff(\WW,\partial);\bZ)$ is not finitely generated.
\end{thm}
\begin{proof}
 As we proved in \Cref{cor4},  the cohomology class $\kappa_c^{\delta}$ for $c\in \mathcal{B}$ are non-torsion classes in $H^*(\Omega_0^{\infty}\bold{MT}\nu^n;\bZ)$.  However for $c$ whose degree is larger than $6n$, the class $\kappa_c^{\delta}$ lives in the kernel of the natural map
  \[
  \begin{tikzpicture}[node distance=3.9cm, auto]
  \node (A) {$H^*(\Omega_0^{\infty}\bold{MT}\nu^n;\bZ)\otimes \bQ$};
  \node (B) [right of=A] {$H^*(\Omega_0^{\infty}\bold{MT}\nu^n; \bQ). $};
  \draw [->] (A) to node {$$}(B);
\end{tikzpicture}
 \]
 Hence, by the universal coefficient theorem, the group $H_{\text{deg}(\kappa_c)-1}(\Omega_0^{\infty}\bold{MT}\nu^n; \bZ)$ cannot be finitely generated.
 \end{proof}

\subsection{On non-vanishing characteristic classes for flat $\W$-bundles}Since \Cref{limit homology} implies that  the cohomology of $\BdDiff(\WW,\partial)$ is related to the cohomology of  the Haefliger classifying space, we can use the nontriviality of classes in $H^*(\mathrm{BS}\Gamma_{2n};\bZ)$ to prove existence of nontrivial characteristic classes for flat $\tW_g$-bundles.
 
 Consider the following commutative diagram
\[
\begin{tikzcd}
\pi_*(\Omega^{\infty}\bold{MT}\nu^n)\otimes \bQ\arrow{r}\arrow{d}&\pi_*(\bold{MT}\nu^n)\otimes\bQ\arrow{d}\\ H_*(\Omega^{\infty}\bold{MT}\nu^n;\bQ)\arrow[two heads]{r}&H_*(\bold{MT}\nu^n;\bQ).
\end{tikzcd}
\] 
The horizontal maps are induced by the suspension map and the vertical maps are induced by the Hurewicz  map.  The top horizontal map is an isomorphism by the definition of the homotopy groups of a spectra  and the right vertical map is also an isomorphism because of the rational Hurewicz theorem (see \cite[Theorem 7.11]{rudyak1998thom}). Therefore, the bottom horizontal map, is surjective. Recall that the map $\nu:\mathrm{BS}\Gamma_{2n}\to \mathrm{BGL}_{2n}^+(\bR)$ is $(2n+2)$-connected. Hence, the map
\[
\begin{tikzpicture}[node distance=2.8cm, auto]
  \node (A) {$\nu^n: {\mathrm B}\Gamma_{2n}\langle n\rangle$};
  \node (B) [right of=A] {$\mathrm{BGL}_{2n}(\bR)\langle n\rangle$};
  \draw [->] (A) to node {$$}(B);
\end{tikzpicture}
\]
is also $(2n+2)$-connected. Thus the group $H^{2n+2}(\mathrm{BGL}_{2n}(\bR)\langle n\rangle; \bQ)$ injects into the group $H^{2n+2}({\mathrm B}\Gamma_{2n}\langle n\rangle; \bQ)$. Now suppose $n\equiv 3 (\textit{mod } 4)$, then using the connectivity of the map $\nu^n$, one deduces that $c=p_{\frac{n+1}{4}}^2$ is nonzero in $H^{2n+2}({\mathrm B}\Gamma_{2n}\langle n\rangle;\bQ)$. Thus the corresponding $\kappa_c^\delta$ is nontrivial in $H^2(\Omega^{\infty}\bold{MT}\nu^n;\bQ)$. Using this observation and  \Cref{limit homology}, we conclude the following theorem. 
\begin{thm}
For $n\equiv 3 (\textit{mod } 4)$, the generalized MMM class $\kappa_c^\delta$ associated to $c=p_{\frac{n+1}{4}}^2$, is nonzero in $H^2(\BdDiff(\WW, \partial);\bQ)$ as $g\geq 7$.
\end{thm}
\begin{rem}
This result is analogous to the surface case which was proved by Kotschick and Morita \cite{kotschick2005signatures}. They showed that $\kappa_1^\delta$ is nonzero on flat surface bundles, hence it is nonzero in $H^2(\BdDiff(\Sigma_{g,1}, \partial);\bQ)$. In the sequel paper \cite{nariman2015stable}, we further study flat surface bundles. 
\end{rem}
Thurston in an unpublished manuscript  (see \cite{thurstonvariation} and \cite {MR769761}), proved that the Godbillon-Vey class $h_1c_1^n\in H^{2n+1}(\overline{\mathrm{B}\Gamma_n};\bZ)$  (for definition of these secondary characteristic classes consult e.g. \cite{pittie1976characteristic}, \cite{bott1972lectures}) varies continuously  on a foliated trivial bundle with fiber dimension $n$. Therefore, in codimension $2n$, we have a surjective map
\[
\begin{tikzpicture}[node distance=2.9cm, auto]
  \node (A) {$H_{4n+1}(\overline{{\mathrm B}\Gamma_{2n}};\bZ)$};
  \node (B) [right of=A] {$\bR,$};
  \draw [->>] (A) to node {$\int h_1c_1^{2n}$}(B);
\end{tikzpicture}
\]
where $\overline{{\mathrm B}\Gamma_{2n}}$ is homotopy fiber of ${\mathrm B}\Gamma_{2n}\rightarrow \mathrm {BGL}_{2n}(\bR)$. This homotopy fiber classifies foliated trivial bundles with fiber dimension $2n$. By Haefliger's theorem, we know $\overline{{\mathrm B}\Gamma_{2n}}$ is at least $(2n+1)$-connected, so there exists a map from $\overline{{\mathrm B}\Gamma_{2n}}$ to ${\mathrm B}\Gamma_{2n}\langle n\rangle$ that makes the following diagram commute
\[
\begin{tikzpicture}[node distance=1.8cm, auto]
  \node (A) {$\overline{{\mathrm B}\Gamma_{2n}}$};
  \node (B) [right of=A] {${\mathrm B}\Gamma_{2n}.$};
  \node (C) [above of=B] {${\mathrm B}\Gamma_{2n}\langle n\rangle$};
  \draw [->] (A) to node {$$}(B);
  \draw [->] (C) to node {$$}(B);
  \draw [->] (A) to node {$$}(C);
\end{tikzpicture}
\]
Therefore, $h_1c_1^{2n}$ is a nonzero class in $H^{4n+1}({\mathrm B}\Gamma_{2n}\langle n\rangle;\bQ)$ and varies continuously. Hence, the composition of the following maps
\[
\begin{tikzpicture}[node distance=3.9cm, auto]
  \node (A) {$H_{2n+1}(\Omega_0^{\infty}\bold{MT}\nu^n;\bQ)$};
  \node (B) [right of=D] {$H_{4n+1}({\mathrm B}\Gamma_{2n}\langle n\rangle;\bQ)$};
  \node (D) [right of=A]{$H_{2n+1}(\bold{MT}\nu^n;\bQ)$};
  \node (C) [right of=B,node distance=3cm]{$\bR$};
  \draw [->>] (B) to node {$\int h_1c_1^{2n}$}(C);
  \draw [->>] (A) to node {$$}(D);
  \draw [->] (D) to node{$\cong$} (B);
\end{tikzpicture}
\]
is surjective.
Using \Cref{limit homology}, one can conclude the above discussion as follows.
\begin{thm}
The following map is  surjective, provided $g\geq 4n+4$
\[
\begin{tikzpicture}[node distance=3.4cm, auto]
  \node (A) {$H_{2n+1}(\BdDiff(\WW,\partial);\bQ)$};
  \node (B) [right of=A] {$\bR$};
  \draw [->>] (A) to node {$\int h_1c_1^{2n}$}(B);
\end{tikzpicture}
\]
i.e. $H_{2n+1}(\BdDiff(\WW, \partial);\bQ)$ as a vector space over rationals has uncountable dimension.
\end{thm}
\begin{rem}\label{rem6.11}
Steve Hurder proved in  \cite[Remark 2.4] {MR769761} that there are at least $3$ continuously varying Godbillon-Vey classes in $H_{4n+1}(\overline{{\mathrm B}\Gamma_{2n}};\bZ)$. Hence, we have at least three continuously varying classes on $\WW$-bundles i.e.
\[
\begin{tikzpicture}[node distance=3.4cm, auto]
  \node (A) {$H_{2n+1}(\BdDiff(\WW,\partial);\bQ)$};
  \node (B) [right of=A] {$\bR^3$};
  \draw [->>] (A) to node {$$}(B);
\end{tikzpicture}
\]

\end{rem}
Since $H_*(\Omega_0^{\infty}\bold{MT}\nu^n;\bQ)$ is a Hopf algebra over $\bQ$, the surjectivity of the following map
\[
\begin{tikzpicture}[node distance=3.9cm, auto]
  \node (A) {$H_{2n+1}(\Omega_0^{\infty}\bold{MT}\nu^n;\bQ)$};
  \node (C) [right of=A,node distance=3cm]{$\bR^3,$};
  \draw [->>] (A) to node {$$}(C);
\end{tikzpicture}
\]
implies that there is a surjective map
\[
\begin{tikzpicture}[node distance=3.9cm, auto]
  \node (A) {$H_{(2n+1)k}(\Omega_0^{\infty}\bold{MT}\nu^n;\bQ)$};
  \node (C) [right of=A,node distance=3cm]{$\bigwedge\nolimits^k_{\bQ}\bR^3,$};
  \draw [->>] (A) to node {$$}(C);
\end{tikzpicture}
\]
where $\bigwedge\nolimits^k_{\bQ}\bR^3$ is the exterior power of $\bR^3$ as a vector space over $\bQ$. Hence, the group $H_{(2n+1)k}(\BdDiff(\WW,\partial);\bQ)$ is also nontrivial for $(2n+1)k\leq (g-3)/2$.
\begin{rem}
We can apply Bowden's idea in \cite{Bowden} to determine the stable homology of $\dDiff(\W,\partial)$ in low homological degrees.  There is a spectral sequence \cite[Theorem 2.3.4]{haller1998perfectness} whose $E^2_{p,q}$ page can be described for $q\leq 3$ as follows
\[
E^2_{p,q}=
\begin{cases}
\bZ & \text{if } p=q=0\\
0& \text{if } q=0, p>0\\
H_p(\W,H_q(\overline{\text{B}\Diff_c(\bR^{2n})})) &\text{if } 0<q\leq3
\end{cases}
\]
it converges to $H_{p+q}(\overline{\text{B}\Diff(\W;\partial)})$ for $p+q\leq 3$. Since there are no differentials in this range, we deduce
\[
H_k(\overline{\text{B}\Diff(\W,\partial)};\bZ)=H_k(\overline{\text{B}\Diff_c(\bR^{2n})};\bZ) \text{ as } k\leq 3
\]
 In particular, $H_0(\overline{\text{B}\Diff(\W,\partial)};\bZ)=\bZ, H_1(\overline{\text{B}\Diff(\W,\partial)};\bZ)=0$.  Using Serre spectral sequence for the following fibration sequence
 \[
  \begin{tikzpicture}[node distance=3.2cm, auto]
  \node (A) {$\overline{\text{B}\Diff(\W,\partial)}$};
  \node (B) [right of=A] {$\BdDiff(\W,\partial)$};
  \node (C) [right of= B ] {$\BDiff(\W,\partial),$};  
  \draw [->] (A) to node {$$}(B);
  \draw [->] (B) to node {$$}(C);
\end{tikzpicture}
 \]
 one can  compute the stable homology of $\BdDiff(\W,\partial)$ in  low homological degrees, using  the computation in \cite{Galatius-Randal-Williams} of the stable homology of $\BDiff(\W,\partial)$. Thus, it is straightforward to see for $g\geq 9$ and $n\geq 3$
 \begin{align*}
 H_1(\BdDiff(\W,\partial);\bZ)&=H_1(\BDiff(\W,\partial);\bZ)\\
 H_1(\BdDiff(\W,\partial);\bQ)&=0\\
 H_2(\BdDiff(\W,\partial);\bQ)&=H_2(\BDiff(\W,\partial);\bQ)\oplus H_2(\overline{\BDiff_c(\bR^{2n})};\bQ)\\
H_3(\BdDiff(\W,\partial);\bQ)&=H_3(\overline{\BDiff_c(\bR^{2n})};\bQ)\\
 \end{align*}
For $g\geq 5$ and $n\geq 3$, the first homology  $H_1(\BDiff(\W,\partial);\bZ)$ has been calculated in \cite[Theorem 1.3]{galatius2015abelian}.
\end{rem}

\bibliographystyle{alpha}
\bibliography{reference}
\end{document}